\documentclass[a4paper, 12pt, reqno]{amsart}
\usepackage[T1]{fontenc}
\usepackage[utf8]{inputenc}
\usepackage[english]{babel}
\usepackage[left=25mm,right=25mm,top=35mm,bottom=35mm]{geometry}
\usepackage{amsmath,amssymb,amsthm}
\usepackage{pgfplots}
\usepackage{pgfplotstable}
\usepackage[margin=0pt,font={small}]{caption}
\usepackage{booktabs}
\usepackage{color,colortbl}
\usepackage[hidelinks]{hyperref}
\usepackage{tabu}
\usepackage{graphicx} 

\newcommand\0{\boldsymbol{0}}
\newcommand{\eps}{\varepsilon}
\newcommand{\x}{\mathbf{x}} 
\newcommand{\rro}{\boldsymbol{\rho}}
\newcommand{\dd}{\boldsymbol{\delta}}

\newcommand{\leb}{\mathcal{L}} 

\newcommand{\kk}{\mathbf{k}}

\newcommand{\ro}{\rho}
\newcommand{\ii}{\mathbf{i}}
\renewcommand{\a}{\mathbf{a}}
\renewcommand{\b}{\mathbf{b}}
\newcommand{\y}{\mathbf{y}} 

\newcommand{\z}{\mathbf{z}} 
\newcommand{\G}{\Gamma}
\newcommand{\N}{\mathbb{N}}
\renewcommand{\P}{\mathbb{P}}
\newcommand{\R}{\mathbb{R}}
\newcommand{\V}{\mathbb{V}}
\newcommand{\X}{\mathbb{X}}
\newcommand{\W}{\mathbb{W}}

\newcommand{\MM}{\mathcal{M}}
\newcommand{\NN}{\mathcal{N}}
\newcommand{\PP}{\mathcal{P}}

\renewcommand{\SS}{\mathcal{S}}
\newcommand{\TT}{\mathcal{T}}

\DeclareMathOperator*{\refine}{refine}

\newcommand{\norm}[3][]{#1\|#2#1\|_{#3}}
\newcommand{\normm}[3][]{#1\|#2#1\|^2_{#3}}
\newcommand{\dpi}{\mathrm{d}\pi}
\newcommand{\dx}{\mathrm{d}x}

\newcommand{\reff}[2]{\stackrel{\eqref{#1}}{#2}}	
\newcommand{\refp}[2]{\stackrel{\phantom{\eqref{#1}}}{#2}}
\newcommand{\Colpts}{\mathcal{Y}} 
\newcommand{\mf}{\kappa} 
\newcommand{\bmf}{\boldsymbol{\kappa}} 
\newcommand{\indset}{\Lambda} 
\newcommand{\markindset}{\Upsilon} 
\newcommand{\rmarg}{{\rm R}} 
\newcommand{\nnu}{\boldsymbol{\nu}} 
\newcommand{\mmu}{\boldsymbol{\mu}} 
\newcommand{\eeps}{\boldsymbol{\eps}} 
\newcommand\1{\boldsymbol{1}}
\newcommand{\Lagr}{I} 
\newcommand{\LagrBasis}[2]{L_{#1}^{#2}} 
\newcommand{\LagrBasisHat}[2]{\widehat L_{#1}^{#2}} 
\newcommand{\scsol}{u_{\bullet}^{\rm SC}} 
\newcommand{\scsolhat}{\widehat u_{\bullet}^{\rm\, SC}} 
\newcommand{\ssolc}{u_0^{\rm semi}} 
\newcommand{\ssol}{u_{\bullet}^{\rm semi}} 
\newcommand{\ssolh}{\hat u_{\bullet}^{\rm semi}} 
\newcommand{\ssoll}[1]{u_{#1}^{\rm semi}}
\newcommand{\ssolhl}[1]{\hat u_{#1}^{\rm semi}}

\pgfplotsset{
every axis/.append style={
font={\fontsize{8pt}{12pt}\selectfont},  
},
tick label style={font=\tiny},
title style={font=\tiny,yshift=-1.5ex},
xlabel style={font=\tiny,yshift=+1.0ex},
ylabel style={font=\tiny,yshift=-1.2ex},
}

\definecolor{myBrown}{rgb}{0.6 0.4 0.2}
\definecolor{myOrange}{rgb}{1.0 0.6 0.2}
\definecolor{myLightGray}{RGB}{235,235,235}
\definecolor{myViolet}{RGB}{153,50,204}
\newtheorem{theorem}{Theorem}

\newtheorem{lemma}[theorem]{Lemma}

\newtheorem{corollary}[theorem]{Corollary}
\newtheorem{algorithm}[theorem]{Algorithm}
\newtheorem{remark}[theorem]{Remark}

\makeatletter
\def\@seccntformat#1{%
  \protect\textup{\protect\@secnumfont
    \ifnum\pdfstrcmp{subsection}{#1}=0 \bfseries\fi
    \csname the#1\endcsname
    \protect\@secnumpunct
  }%
}
\makeatother
\usepackage{fancyhdr}
\advance\footskip0.5cm
\definecolor{otherblue}{rgb}{0,0.3,0.6}
\newcommand\rev[1]{{\color{black}#1}}

\title{Convergence analysis of the adaptive stochastic collocation finite element method}
\author{Alex Bespalov}
\address{School of Mathematics, University of Birmingham, Edgbaston, Birmingham B15 2TT, UK}
\email{a.bespalov@bham.ac.uk}
\author{Andrey Savinov}
\address{School of Mathematics, University of Birmingham, Edgbaston, Birmingham B15 2TT, UK}
\email{avs296@student.bham.ac.uk}



\thanks{\emph{Acknowledgements.}
The work of the first author was supported by the EPSRC grant EP/W010925/1.
}
\date{\today}

\begin{document}

\begin{abstract}
This paper is focused on the convergence analysis of an adaptive stochastic collocation algorithm for the stationary diffusion equation
with parametric coefficient.
The algorithm employs sparse grid collocation in the parameter domain alongside finite element approximations in the spatial domain,
and adaptivity is driven by recently proposed parametric and spatial a posteriori error indicators. 
We prove that for a general diffusion coefficient with finite-dimensional parametrization, the algorithm drives the underlying
error estimates to zero.
Thus, our analysis covers problems with affine and nonaffine parametric coefficient dependence.
\end{abstract}

\maketitle
\thispagestyle{fancy}

\section{Introduction}
Sparse grid stochastic collocation is an established and well-studied computational method
for solving high-dimensional parametric partial differential equations (PDEs)
that are ubiquitous in uncertainty quantification models.
The sparsity of the underlying set of collocation points is critical for this task even for moderately high-dimensional problems,
as basic tensor-product grids of collocation points yield approximations suffering from the curse of dimensionality;
see, for example,\rev{~\cite{BabuskaNT07, NobileTW08a, NobileTW08b, Bieri_11_SCC, BeckTNT_12_OPA,
NobileTamelliniTempone2016, ErnstST_18_CSC, DNSZ2023}}.
Furthermore, in a typical setting of a high-dimensional parametric PDE, the solution is anisotropic in the parameter domain,
calling for adaptive enrichment of sparse grid collocation points.

Adaptively generated sparse grids trace back to the work of Gerstner and Griebel~\cite{GerstnerGriebel2003} on high-dimensional \rev{quadrature}.
Their ideas have found successful applications to collocation methods for parametric PDEs,
see, for example,~\cite{ChkifaCohenSchwab2014, NobileTamelliniTeseiTempone2016},
where parametric adaptivity is driven by heuristic error indicators that require solving additional PDEs.
An alternative approach, proposed in~\cite{GuignardN18}, is based on a posteriori error estimation.
Here, a reliable residual-based a posteriori error estimator is derived to control two distinct sources of discretization error
arising from parametric (sparse grid collocation) and spatial (finite element) components of approximations.
Crucially, this error estimator avoids the solution of additional PDEs;
it is also localizable and, thus, can be readily used in an adaptive algorithm.
In particular, the parametric component of the error estimator is used in~\cite{GuignardN18} to design an algorithm
that generates adaptive sparse grid (semidiscrete) approximations.
The convergence analysis of a modified version of the adaptive algorithm in~\cite{GuignardN18} is performed in~\cite{eest22}.
In~\cite{FeischlS21}, the authors extended the adaptive algorithm proposed in~\cite{GuignardN18} to include spatial (finite element) adaptivity
and proved convergence of the resulting `fully adaptive' algorithm.

It is important to note that the a posteriori error estimation framework developed in~\cite{GuignardN18}
and, hence, the adaptive algorithms in~\cite{GuignardN18, eest22, FeischlS21}, are inherently restricted to parametric PDEs whose inputs
have \emph{affine} dependence on parameters; cf.~\cite[section~4]{GuignardN18}.
In our recent work~\cite{Bespalov22}, we proposed a novel a posteriori error estimation strategy
for finite element-based sparse grid stochastic collocation approximations.
This error estimation strategy is applicable
to general elliptic parametric PDEs
with either affine or nonaffine parametric dependence of inputs.
It is similar in spirit to the hierarchical error estimation framework proposed in the context of stochastic Galerkin finite element methods,
see~\cite{bps14, bs16, bprr18, bprr18+, bx2020}.

In this contribution, we bridge a gap in the existing theory of adaptive algorithms for stochastic collocation finite element methods (SC-FEMs)
by extending the convergence analysis in~\cite{eest22, FeischlS21} to a broader class of parametric elliptic PDEs that covers
problems with \emph{nonaffine} parametric coefficients.
Specifically, we study convergence of an adaptive algorithm guided by reliable a posteriori error estimates
and the associated error indicators proposed in~\cite{Bespalov22}.
We note that the adaptive algorithm considered in this work is slightly different from the one proposed in~\cite{Bespalov22};
the difference lies in how the algorithm performs parametric marking and enrichment (see Remark~\ref{remark:param:marking}).
Our main result in Theorem~\ref{theorem:main} shows that the modified adaptive algorithm generates
SC-FEM approximations, with the corresponding sequence of error estimates converging to zero.
A key ingredient of our analysis is a certain summability property of Taylor coefficients for semidiscrete (finite element) approximations.
We prove that this summability property holds in the case of affine-parametric coefficients
satisfying the uniform ellipticity assumption (see Lemma~\ref{lemma:summability:affine}).
Furthermore, for PDEs with general parametric coefficients, we show that the assumptions guaranteeing the analyticity of the exact solution
in the parameter domain also ensure the required summability property of Taylor coefficients for semidiscrete approximations
(see Lemma~\ref{lemma:summability:general} and Remark~\ref{rem:summability:general}).

The outline of the paper is as follows.
After introducing the parametric model problem in section~\ref{sec:problem},
we set up its stochastic collocation discretization in section~\ref{sec:scfem}.
Section~\ref{sec:summability:property} focuses on the summability property of Taylor coefficients for semidiscrete (finite element) approximations.
In section~\ref{sec:error:estimate}, we recall the main components of the a posteriori error estimation strategy
developed in~\cite{Bespalov22} and present the adaptive SC-FEM algorithm.
Sections~\ref{sec:param:convergence} and~\ref{sec:spat:convergence} focus on proving convergence of parametric and spatial
error estimates for SC-FEM approximations generated by the adaptive algorithm.
In section~\ref{sec:main:results}, we formulate and prove the main result of this work.
The results of numerical experiments are presented and discussed in section~\ref{sec:numerics}.

\section{Parametric model problem} \label{sec:problem}

Let \rev{$D \subset \R^d \; (d=2,3)$} be a bounded Lipschitz domain with \rev{polytopal} boundary $\partial D$;
we will refer to $D$ as the \emph{spatial} domain.
Let us also introduce the \emph{parameter} domain
$\Gamma := \G_1 \times \G_2 \ldots \times \G_M \subset \R^M$,
where $M \in \N$ and each $\G_m$ ($m = 1,\ldots, M$) is a bounded interval in~$\R$.
Let $\pi(\y) := \prod_{m=1}^M \pi_m(y_m)$ be a probability measure on $(\G,\mathcal{B}(\G))$;
here,
$\mathcal{B}(\G)$ is the Borel $\sigma$-algebra on $\G$, and
$\pi_m$ denotes a probability measure on $(\G_m,\mathcal{B}(\G_m))$ for $m = 1,\ldots,M$.

We consider the following parametric elliptic problem:
find $u : \overline D \times \G \to \R$ satisfying
\begin{equation} 
\label{eq:pde:strong}
\begin{aligned}
-\nabla \cdot (a(\cdot, \y)\nabla u(\cdot, \y))
&= f
&& \text{in $D$},\\ 
u(\cdot, \y) &= 0  && \text{on $\partial D$} 
\end{aligned}
\end{equation} 
$\pi$-almost everywhere on $\G$ (i.e., almost surely).
Here, the forcing term $f \in L^2(D)$ is deterministic and
the coefficient $a$ is a random field on $(\G,\mathcal{B}(\G),\pi)$ over $L^\infty(D)$.
We assume that the coefficient $a$ is positive and bounded, i.e.,
\begin{equation} \label{eq:amin:amax}
   0 < a_{\min} \le \operatorname*{ess\;inf}_{x \in D} a(x,\y) \le 
   \operatorname*{ess\;sup}_{x \in D} a(x,\y) \le a_{\max} < \infty \quad
   \text{$\pi$-a.e. on $\G$}
\end{equation}
with some positive constants $a_{\min}$, $a_{\max}$.
This assumption implies the norm equivalence for the Sobolev space on the spatial domain:
for any $v \in \X := H^1_0(D)$ there holds
\begin{equation} \label{eq:norm:equiv}
   a_{\min} \|\nabla v\|_{L^2(D)}^2 \le
   \| a(\cdot,\y) \nabla v \|_{L^2(D)}^2 \le
   a_{\max} \|\nabla v\|_{L^2(D)}^2\quad
   \text{$\pi$-a.e. on $\G$}.
\end{equation}
For the purpose of finding the numerical solution to the parametric problem~\eqref{eq:pde:strong},
we write it in the following weak form:
given $f \in L^2(D)$, find $u : \G \to \X$ such that
\begin{align} \label{eq:pde:weak}
   \int_D a(x, \y) \nabla u(x,\y) \cdot \nabla v(x) \dx = \int_D f(x) v(x) \dx 
   \quad \forall v \in \X,\ \text{$\pi$-a.e. on $\G$}.
\end{align}
The above assumptions on the problem data
ensure 
the existence and uniqueness of the 
solution $u$
in the Bochner space $\V := L_\pi^p(\G; \X)$ for any $p \in [1, \infty]$;
see~\cite[Lemma~1.1]{BabuskaNT07} for details.

\section{Stochastic collocation finite element method} \label{sec:scfem}

For the numerical solution of problem~\eqref{eq:pde:strong} we apply
the stochastic collocation finite element method.
Let us recall the main ideas, including the construction of the underlying approximation spaces.

We denote by $\TT_\bullet$ a mesh on the spatial domain $D$
(i.e., a conforming \rev{partition of $D$ into compact nondegenerate simplices $T \in \TT_\bullet$}),
and let $\NN_\bullet$ denote the set of vertices of~$\TT_\bullet$.
\rev{Here and throughout the paper, we use $\bullet$ as a placeholder for the iteration counter;
see, e.g., $\TT_\ell$ in Algorithm~\ref{algorithm}.}
For mesh refinement, we employ newest vertex bisection (NVB); see, e.g., \cite{stevenson,kpp}. 
We assume that any mesh $\TT_\bullet$ employed for the spatial discretization
is obtained by (uniform or local) refinement of a given (coarse) initial mesh $\TT_0$.
For the numerical solution of~\eqref{eq:pde:weak}, 
we employ the space $\X_{\bullet}$ of continuous piecewise linear functions,
\begin{equation*}
 \X_\bullet := \SS^1_0(\TT_\bullet) :=
 \{v \in \X : v \vert_T \text{ is affine for all } T \in \TT_\bullet \} \subset \X = H^1_0(D).
\end{equation*}
In particular, $\X_0 := \SS^1_0(\TT_0).$
The standard basis of $\X_\bullet$ is given by
$\{ \varphi_{\bullet,\xi} : \xi \in \NN_\bullet \setminus \partial D \}$, where
$\varphi_{\bullet,\xi}$ denotes the hat function associated with the vertex $\xi \in \NN_\bullet$.

Let $\widehat\TT_\bullet$ be the mesh obtained by uniform NVB refinement of $\TT_\bullet$\rev{, i.e.,
$\widehat\TT_\bullet$ is the coarsest mesh obtained from $\TT_\bullet$ such that:
(i)~for $d=2$, all elements in~$\TT_\bullet$ are refined by three bisections (see, e.g.,~\cite[Figure~1]{bprr18});
(ii)~for $d=3$,
all elements are refined as described in~\cite[section~2.1]{egp18+} and illustrated in~\cite[Figures~2 and~3]{egp18+}}.
Then, $\widehat\NN_\bullet$ denotes the set of vertices of $\widehat\TT_\bullet$,
and $\NN_\bullet^+ := (\widehat\NN_\bullet \setminus \NN_\bullet) \setminus \partial D$
is the set of new interior vertices created by this refinement of $\TT_\bullet$.
The finite element space associated with $\widehat\TT_\bullet$ is denoted as
$\widehat\X_\bullet := \SS^1_0(\widehat\TT_\bullet)$,
and
$\{ \widehat\varphi_{\bullet,\xi} : \xi \in \widehat\NN_\bullet \setminus \partial D \}$ is the corresponding  basis of hat functions.

Let $\z$ be a fixed point in $\G$.
We denote by $u_{\bullet \z} \in \X_{\bullet}$ the Galerkin finite element approximation satisfying
\begin{align} \label{eq:sample:fem}
   \int_D a(x, \z) \nabla u_{\bullet \z}(x) \cdot \nabla v(x) \dx = \int_D f(x) v(x) \dx 
   \quad \forall v \in \X_{\bullet}.
\end{align} 
Hence, given a finite set $\Colpts_\bullet$ of collocation points in $\G$,
the SC-FEM approximation of the solution $u$ to parametric problem~\eqref{eq:pde:strong} is given by
\begin{equation} \label{eq:scfem:sol}
   \scsol(x, \y) := \sum\limits_{\z \in \Colpts_\bullet} u_{\bullet \z}(x) \LagrBasis{\bullet \z}{}(\y),
\end{equation}
where $\{\LagrBasis{\bullet \z}{}(\y), \z \in \Colpts_\bullet \}$ is a set of
\rev{multivariate} Lagrange basis functions constructed for the set of collocation points $\Colpts_\bullet$ and satisfying
$\LagrBasis{\bullet \z}{}(\z') = \delta_{\z\z'}$, $\forall \z,\, \z' \in \Colpts_\bullet$.

Note that the SC-FEM solution considered in this work
follows the so-called \emph{single-level} construction that employs the same finite element space $\X_{\bullet}$
for all collocation points $\z \in \Colpts_\bullet$ (cf.~\cite{BabuskaNT07,NobileTW08a,GuignardN18,Bespalov22}).
This is in contrast to the \emph{multilevel} SC-FEM approximations that allow
$\X_{\bullet \z} \not= \X_{\bullet \z'}$ for $\z \not= \z'$; see, e.g.,~\cite{LangSS20,FeischlS21,BS23}.

In the context of the numerical solution of high-dimensional parametric problems,
the state-of-the-art stochastic collocation methods employ the nodes of \emph{sparse grids} as collocation points.
We briefly describe the construction of sparse grids in the next section.

\subsection{Sparse grid interpolation} \label{sec:sparse:grids}

Since any finite interval in $\R$ can be mapped to $[-1, 1]$ via appropriate linear transformation,
we can assume without loss of generality that $\G_1 = \G_2 = \ldots = \G_M = [-1, 1]$.
The construction of a sparse grid $\Colpts_\bullet \subset \Gamma = [-1,1]^M$ hinges on three ingredients:
\begin{itemize}

\item
a family of \emph{nested sets} of 1D nodes on $[-1,1]$
(in this work, we will consider the nested sets of Leja points and Clenshaw--Curtis (CC) quadrature points);

\item
a strictly increasing function $\mf : \N_0 \to \N_0$ satisfying $\mf(0) = 0$, $\mf(1) = 1$
(e.g., 
$\mf(i) = i$ for Leja points and
$\mf(i) = 2^{i-1}+1$, $i > 1$ for CC
nodes with the doubling rule).

\item
a monotone finite set $\indset_\bullet \,{\subset}\, \N^M$ of multi-indices;
specifically, 
$\indset_\bullet \,{=}\, \{ \nnu \,{=}\, [\nu_1,\ldots,\nu_M] : \nu_m \in \N,\, m = 1,\ldots,M \}$ is such 
that $\#\indset_\bullet < \infty$ and
\[
   \nnu \in \indset_\bullet \Longrightarrow \nnu - \eeps_m \in \indset_\bullet \quad
   \forall\,m=1,\ldots,M \text{ such that } \nu_m >1,
\]
where $(\eeps_m)_i = \delta_{mi}$ for all $i=1,\ldots,M$.
Note that the monotonicity property of $\indset_\bullet$ implies that $\1 = [1, 1,\dots,1] \in \indset_\bullet$;
\end{itemize}

Now, for each $\nnu \in \indset_\bullet$, the set of nodes along the $m$-th coordinate axis in $\R^M$
is given by the set $\Colpts_m^{\mf(\nu_m)}$ such that $\#\Colpts_m^{\mf(\nu_m)} = \mf(\nu_m)$,
and we define
\[
   \Colpts^{\,(\nnu)} := \Colpts_1^{\mf(\nu_1)} \times \Colpts_2^{\mf(\nu_2)} 
   \times \ldots \times \Colpts_M^{\mf(\nu_M)}.
\]
For a given index set $\indset_\bullet$, the sparse grid $\Colpts_\bullet$
of collocation points on $\G$ is defined~as
\[
   \Colpts_\bullet = \Colpts_{\indset_\bullet} := \bigcup_{\nnu \in \indset_\bullet} \Colpts^{\,(\nnu)}.
\]
Let $\Lagr_m^{\mf(\nu_m)} : C^0([-1,1];\X) \to \P_{\mf(\nu_m)-1}([-1,1];\X)$ be
the standard Lagrange interpolation operator associated with the set of 1D nodes $\Colpts_m^{\mf(\nu_m)} \subset [-1,1]$.
Here, $\P_{q}$ is the set of univariate polynomials of degree at most $q \in \N_0$.
Setting $\Lagr_m^{0} = 0$ for all $m = 1,\ldots,M$, we define 1D detail operators
\[
   \Delta_m^{\mf(\nu_m)} := \Lagr_m^{\mf(\nu_m)} - \Lagr_m^{\mf(\nu_m-1)}.
\]
Now, the sparse grid collocation operator associated with the sparse grid $\Colpts_{\indset_\bullet}$ is defined~as
\begin{equation} \label{eq:S:def}
   S_\bullet = S_{\indset_\bullet} := \sum\limits_{\nnu \in \indset_\bullet} \Delta^{\bmf(\nnu)},
\end{equation}
where $\bmf(\nnu) := [\mf(\nu_1), \ldots, \mf(\nu_M)]$ and
$\Delta^{\bmf(\nnu)} := \bigotimes_{m=1}^{M} \Delta_m^{\mf(\nu_m)}$ is called the hierarchical surplus operator.

The nestedness of 1D node sets and the monotonicity of the index set $\indset_\bullet$ 
ensure the interpolation property for the operator $S_{\indset_\bullet}$ (cf.~\cite{BNTL11}), i.e.,
\begin{equation} \label{eq:interpolation}
   S_{\indset_\bullet}: C^0(\G;\X) \to \P_{\indset_\bullet}(\G;\X)
   \text{\; is such that \;}
   S_{\indset_\bullet} v(\z) = v(\z)\ \; \forall\,\z \in \Colpts_{\indset_\bullet}, 
\end{equation}
where $\P_{\indset_\bullet} := \bigoplus_{\nnu \in \indset_\bullet} \P_{\bmf(\nnu)-\1}$ with $\P_{\bmf(\nnu)-\1} := \bigotimes_{m=1}^M \P_{\mf(\nu_m)-1}$.
Therefore, the SC-FEM solution defined by~\eqref{eq:scfem:sol} can be written as
\begin{equation} \label{eq:scfem:sol:S-form}
   \scsol(x,\y) = [S_{\bullet} U_\bullet] (x,\y) = \sum\limits_{\z \in \Colpts_\bullet} u_{\bullet \z}(x) \LagrBasis{\bullet \z}{}(\y)
\end{equation}
with a function $U_{\bullet} : \G \to \X_\bullet$ satisfying $U_{\bullet}(\z) = u_{\bullet \z}$ for all $\z \in \Colpts_\bullet$.

Let $\W \in \big\{ \X_0, \X_\bullet, \widehat\X_\bullet \big\}$.
We introduce a semidiscrete approximation $w : \Gamma \to \W$ such that $w(\cdot,\y) \in \W$ satisfies
\begin{align} \label{eq:semidiscrete}
   \int_D a(x, \y) \nabla w(x,\y) \cdot \nabla v(x) \dx = \int_D f(x) v(x) \dx 
   \quad \forall v \in \W,\ \text{$\pi$-a.e. in $\G$}
\end{align}
(for each $\W \in \big\{ \X_0, \X_\bullet, \widehat\X_\bullet \big\}$, the corresponding
semidiscrete approximation satisfying~\eqref{eq:semidiscrete} will be denoted by
$\ssolc,\, \ssol$ and $\ssolh$, respectively).
We assume that the following two representations of $w$ hold (cf.~\cite[eq.~(12)]{GuignardN18}):
\begin{equation} \label{eq:poly:representation}
    w(x,\y) := \sum\limits_{\ii \in \N_0^M} w_{\ii}(x) P_{\ii}(\y) = \sum\limits_{\nnu \in \N^M}\rev{\left[ \Delta^{\bmf(\nnu)} w\right]}(x,\y)
    \quad
    \text{$\pi$-a.e. on $\G$},
\end{equation}
where $P_{\ii}(\y) = \prod_{m=1}^M y_m^{i_m}$ and
\begin{equation} \label{eq:Taylor:coeff}
   w_{\ii}(x) = \frac{1}{\ii!} \frac{\partial^{\ii} w(x, \y)}{\partial \y^{\ii}}\Big\vert_{\y = \0},\quad
   \ii \in \N^M_0
\end{equation}
are the Taylor coefficients.
The following summability property of the Taylor coefficients $w_{\ii}(x)$ will play a key role in our analysis:
there exists $\rro = [\ro_1, \dots, \ro_M] > \1$
such that
\begin{equation} \label{eq:summability:property}
    \left(  \rro^{\ii} \norm[\big]{w_{\ii}}{\X} \right)_{\ii \in \N_0^M} \in l^2(\N_0^M)
    \quad \text{and}\quad
    \sum\limits_{\ii \in \N_0^M} \rro^{2\ii} \normm[\big]{w_{\ii}}{\X} \le C < \infty,
\end{equation}
where $C$ is independent of the underlying finite element space.
Hereafter, for two vectors $\a = [a_1,\ldots,a_M] \in \R^M$ and $\b = [b_1,\ldots,b_M] \in \R^M$, we use the notation
$ \a^\b := \prod_{m=1}^M a_m^{b_m}$ and $\a\b := \prod_{m=1}^M a_mb_m$;
we also write $\a > \b$ 
iff $a_m > b_m$ 
for all $m=1, \dots, M$;
furthermore, for $\ii \in \N_0^M$, we denote by $\ii! = \prod_{m=1}^M i_m!$ the multivariable factorial.

In the next section, we establish the conditions on the problem data
(specifically, on the coefficient $a(\cdot,\y)$ in~\eqref{eq:pde:strong}) that guarantee the summability property~\eqref{eq:summability:property}.

\section{The summability property of Taylor coefficients} \label{sec:summability:property}

We start with the case of the diffusion coefficient $a(\cdot,\y)$ having affine dependence on the~parameters.

\begin{lemma} \label{lemma:summability:affine}

Suppose that the diffusion coefficient has affine representation, i.e.,
\[
   a(x, \y) = a_0(x) + \sum\limits_{m=1}^M a_m(x) y_m\quad \text{for all $x \in D$ and $\y \in \G = [-1; 1]^M$}.
\]
If the expansion coefficients $a_m \in L^{\infty}(D)$, $m = 0,1,\ldots,M$,
satisfy the uniform ellipticity assumption, i.e.,
\[
   \exists\,r>0 \text{ such that }\ \sum\limits_{m=1}^M |a_m(x)| \le a_0(x) - r\quad \forall x \in D,
\]
then inequalities~\eqref{eq:amin:amax} hold and the Taylor coefficients $w_{\ii}(x)$ given by~\eqref{eq:Taylor:coeff}
satisfy the summability property~\eqref{eq:summability:property}.

\end{lemma}
\begin{proof}
It is easy to see that the uniform ellipticity assumption implies~\eqref{eq:amin:amax};
in particular, there holds
\begin{equation*}
   r \le a_0(x) - \sum\limits_{m=1}^M |a_m(x)| \le |a(x, \y)| \le a_0(x) + \sum\limits_{m=1}^M |a_m(x)| \le 2a_0(x) - r
\end{equation*}
and therefore,
\begin{equation*} 
   r \le a_{\min}\quad
   \text{and}\quad
   a_{\max} \le 2 \operatorname*{ess\;\sup}_{x\in D} a_0(x) - r.
\end{equation*}
Now, in order to prove the summability property~\eqref{eq:summability:property}
in the current setting of the affine representation satisfying the uniform ellipticity assumption,
let us define
\[
    \alpha := 1 - \frac{a_{\min}}{\norm{a_0}{L^\infty(D)}} \in (0,1).
\]
Then for any $\rro = [\ro_1, \dots, \ro_M]$ with $1<\ro_m<\alpha^{-1}$ ($m=1, \dots, M$), there holds
\begin{align*}
   \delta := \norm[\bigg]{\frac{\sum_{m=1}^{M} \ro_m |a_m|}{a_0}}{L^\infty(D)}
   & <
   \alpha^{-1} \, \norm[\bigg]{\frac{\sum_{m=1}^{M} |a_m|}{a_0}}{L^\infty(D)} =
   \alpha^{-1} \, \norm[\bigg]{1 - \frac{a_0 - \sum_{m=1}^{M} |a_m|}{a_0}}{L^\infty(D)}
   \\[4pt]
   & \le
   \alpha^{-1} \, \bigg( 1 - \frac{a_{\min}}{\norm{a_0}{L^\infty(D)}} \bigg) = 1.
\end{align*}
This shows that the weighted uniform ellipticity assumption from~\cite[Lemma~2.1 and Theorem~2.2]{bcm17}
is satisfied (cf.~\cite[eq.~(2.20)]{bcm17}).
Repeating the arguments in the proof of~\cite[Lemma~2.1 and Theorem~2.2]{bcm17} for the Taylor coefficients
$w_{\ii}$ of the semidiscrete approximation $w$ (rather than for the Taylor coefficients of the exact solution~$u$) proves that
$\left(\rro^{\ii} \norm[\big]{w_{\ii}}{\X} \right)_{\ii \in \N^M_0} \in l^2(\N^M_0)$ and
\[
   \sum\limits_{\ii \in \N^M_0} \left( \rro^{\ii} \norm[\big]{w_{\ii}}{\X} \right)^2 \le
   \frac{(2-\delta) \norm{a_0}{L^\infty}}{(2-2\delta) \big(\!\operatorname*{ess\;inf}_{x \in D} a_0(x)\big)^3} \, \norm{f}{L^2(D)} =: C < \infty,
\]
where $C$ is independent of the underlying finite element mesh; cf.~\cite[eq.~(2.22)]{bcm17}.
This completes the~proof.
\end{proof}

Next, inspired by the analysis in~\cite{BabuskaNT07}
for a general diffusion coefficient $a(\cdot,\y)$,
we identify the assumptions on $a(\cdot,\y)$ that ensure the summability property~\eqref{eq:summability:property}.

\begin{lemma} \label{lemma:summability:general}
Suppose that inequalities~\eqref{eq:amin:amax} hold for the diffusion coefficient $a(x, \y)$
and assume that for every $\y \in \G$, the derivatives of $a(x, \y)$
with respect to parameters satisfy the following inequalities:
\begin{equation} \label{eq:nonaffine}
   \left\lVert a^{-1}(\cdot, \y) \, \frac{\partial^\kk a(\cdot, \y)}{\partial \y^\kk} \right\rVert_{L^{\infty} (D)} \le
   (2\dd)^{-\kk} \, \kk!
   \quad \forall\,\kk \in \N_0^M\setminus\{ \0\}
\end{equation}
with some vector $\dd = [\delta_1, \dots, \delta_M] > \1.$
Then the Taylor coefficients $w_{\ii}(x)$ given by~\eqref{eq:Taylor:coeff}
satisfy the summability property~\eqref{eq:summability:property}.
\end{lemma}

\begin{proof}
Recall that $\W \in \big\{ \X_0, \X_\bullet, \widehat\X_\bullet \big\}$.
Taking into account the assumptions~\eqref{eq:nonaffine} on the diffusion coefficient,
we can repeat the proof of~\cite[Lemma~3.2]{BabuskaNT07} for
the semidiscrete approximation $w(\cdot,\y) \in \W$ satisfying~\eqref{eq:semidiscrete}
to make the following two conclusions:
\begin{itemize}
\item
the semidiscrete approximation $w$ as a function of $\y$ admits an analytic extension in the region
$\Sigma(\G, \boldsymbol{\sigma}) =
  \left\{ \boldsymbol{\zeta} \in \mathbb{C}^M, \; \text{dist}(\zeta_m, \G_m)\le \sigma_m, \; m=1,\dots, M\right\}
$
for some $\boldsymbol{\sigma} = [\sigma_1, \dots, \sigma_M]$ such that $\boldsymbol{1} < \boldsymbol{\sigma} < \dd$;
\item
there holds
$\max\limits_{\boldsymbol{\zeta} \in \Sigma(\G, \boldsymbol{\sigma})} \norm{w(\cdot, \boldsymbol{\zeta})}{\X} \le C_{\rm reg}$
with a positive constant $C_{\rm reg}$ that depends on the problem data and is independent of the discretization in the spatial domain
(in fact, $C_{\rm reg}$ is exactly the same as given in the proof of Lemma~3.2 in~\cite{BabuskaNT07}).
\end{itemize}
Hence, using Cauchy's integral formula in each $\y$-coordinate, we obtain
for any $\ii \in \N_0^M$ and for all $\y \in \G$:
\begin{equation*}
    \frac{\partial^{\ii} w(x, \y)}{\partial \y^{\ii}} =
    \frac{\ii!}{(2\pi i)^M} \int_{\partial B_{\sigma_M}(y_M)}  \dots \int_{\partial B_{\sigma_1}(y_1)}
    \frac{w(x, \boldsymbol{\zeta})}{(\zeta_1 - y_1)^{i_1+1} \dots (\zeta_M - y_M)^{i_M+1}} \;
    \mathrm{d}\zeta_1 \dots \mathrm{d}\zeta_M,
\end{equation*}
where $\partial B_{\sigma_m}(y_m) \subset \mathbb{C}$ for each $m=1, \dots, M$
denotes the circle of radius $\sigma_m$ centered at $y_m \in \G_m$.
Thus, we can estimate the $\X$-norm of the Taylor coefficient of $w$ as follows:
\begin{equation*}
    \norm[\big]{w_{\ii}}{\X} =
    \norm[\bigg]{ \frac{1}{\ii!} \frac{\partial^{\ii} w(x, \y)}{\partial \y^{\ii}}\Big\vert_{\y = \0} }{\X} \leq
    \boldsymbol{\sigma}^{-\ii}
    \max\limits_{\boldsymbol{\zeta} \in \Sigma(\G, \boldsymbol{\sigma})} \norm{w(\cdot, \boldsymbol{\zeta})}{\X} \leq
    C_{\rm reg} \, \boldsymbol{\sigma}^{-\ii}\quad \forall\,\ii \in \N_0^M.
\end{equation*}
Therefore, there exists a vector $\rro = [\ro_1,\ldots,\ro_M]$ such that
$\delta_m > \sigma_m> \ro_m > 1$ for all $m=1, \dots, M$ and 
$\left(\rro^{\ii} \norm[\big]{w_{\ii}}{\X} \right)_{\ii \in \N^M_0} \in l^2(\N^M_0)$.
This proves~\eqref{eq:summability:property} as required.
\end{proof}

\begin{remark} \label{rem:summability:general}
For the statement of Lemma~{\rm \ref{lemma:summability:general}} to hold,
the assumption on $a(x,\y)$ in~\eqref{eq:nonaffine}
is required only for $\y = \0$, i.e., it is sufficient to assume in Lemma~{\rm \ref{lemma:summability:general}} that
\begin{equation*} 
  \left\lVert
  a^{-1}(\cdot, \0) \, \frac{\partial^\kk a(\cdot, \y)}{\partial \y^\kk}\Big\vert_{\y = \0}
  \right\rVert_{L^{\infty} (D)}  \le
    (2\dd)^{-\kk} \, \kk!
   \quad \forall\,\kk \in \N_0^M\setminus \{\0\}
\end{equation*}
with some vector $\dd > \1$, and the proof can be modified accordingly.
However, if~\eqref{eq:nonaffine} holds true for every $\y \in \G$,
the exact solution $u(\cdot,\y)$ of~\eqref{eq:pde:weak}
admits an analytic extension into a region in $\mathbb{C}^M$ due to~\cite[Lemma~3.2]{BabuskaNT07}.
Importantly, this analyticity property also holds for
semidiscrete solutions $\ssol(\cdot,\y) \in \X_\bullet$ and $\ssolh(\cdot, \y)  \in \widehat\X_\bullet$
satisfying~\eqref{eq:semidiscrete} with $\W = \X_\bullet$ and $\W = \widehat\X_\bullet$, respectively.
We will exploit this fact in the proof of Theorem~{\rm \ref{theorem:parametic:alt}} below.
\end{remark}

\section{Error estimates, error indicators and adaptive algorithm} \label{sec:error:estimate}

In this section, we briefly recall the a posteriori error estimation strategy for SC-FEM approximations developed in~\cite{Bespalov22}
as well as the associated error indicators that steer adaptive refinement;
we refer to~\cite[section~4]{Bespalov22} for full details.

We denote by $\norm{\cdot}{}$ the norm in the Bochner space $\V = L^p_{\pi}(\G,\X)$ for a fixed $1 \le p \le \infty$
and we define $\norm{\cdot}{\X} := \norm{\nabla\cdot}{L^2(D)}$.
We set $p = 2$ when computing the norms in $\V = L^p_{\pi}(\G,\X)$ in practice.
The error estimation strategy developed in~\cite{Bespalov22} employs a hierarchical construction
(see, e.g.,~\cite[Chapter~5]{AinsworthOden00}).
This construction relies on
\rev{an \emph{enhanced} SC-FEM approximation, denoted by $\scsolhat$, and allows one to independently estimate
the spatial and parametric contributions to the overall discretisation error $u - \scsol$.
For the specific construction of $\scsolhat$, we follow~\cite[Remark~1]{Bespalov22}:
\[
   \scsolhat := S_\bullet \widehat U_\bullet + \Big( \widehat S_\bullet \widetilde U_{\bullet,0} - S_\bullet U_{\bullet,0} \Big),
\]
where
\[
   \widehat S_\bullet = S_{\widehat\indset_\bullet} := \sum\limits_{\nnu \in \widehat\indset_\bullet} \Delta^{\bmf(\nnu)},
\]
\[
   \text{$\widehat U_{\bullet} : \G \to \widehat\X_\bullet$ satisfies $\widehat U_{\bullet}(\z) = \widehat u_{\bullet \z}$ for all $\z \in \Colpts_\bullet$},
\]
\[
   \text{$\widetilde U_{\bullet,0} : \G \to \X_0$ is such that
   $\widetilde U_{\bullet,0}(\z') = u_{0\z'}$ for all $\z' \in \widehat\Colpts_\bullet$},
\]
\[
   \text{$U_{\bullet,0} : \G \to \X_0$ is such that $U_{\bullet,0}(\z) = u_{0\z}$ for all $\z \in \Colpts_\bullet$}.
\]
Here, $\widehat u_{\bullet \z} \,{\in}\, \widehat\X_{\bullet}$ denotes the enhanced Galerkin solution satisfying~\eqref{eq:sample:fem}
for all $v \,{\in}\, \widehat\X_{\bullet}$,
the functions $u_{0 \z'},\, u_{0 \z} \in \X_0$ solve~\eqref{eq:sample:fem} with $\X_\bullet$ replaced by $\X_0$, and
$\widehat\Colpts_\bullet$ is the set of collocation points generated by the enriched index set $\widehat\indset_\bullet$
that is obtained from the current index set $\indset_\bullet$ by adding its reduced margin
\begin{equation} \label{eq:reduced:margin}
   \rmarg_{\bullet} = \rmarg({\indset_\bullet}) :=
   \{ \nnu \in \N^{M} \backslash \indset_\bullet :
   \nnu {-} \eeps_m \in \indset_\bullet \text{ for all } m = 1, \dots, M \text{ such that } \nu_m \,{>}\, 1 \}.
\end{equation}
}%
Under the assumption that $\scsolhat$ reduces the discretization error,
i.e.,
\begin{equation} \label{eq:saturation}
   \norm{u - \scsolhat}{} \le q_{\rm sat} \, \norm{u - \scsol}{}
\end{equation}
with a constant $q_{\rm sat} \in (0,1)$ independent of discretization parameters,
the error estimate in the SC-FEM approximation $\scsol$
is given by the sum of spatial and parametric contributions
(see equations~(22),~(23) in~\cite{Bespalov22}):
\begin{equation} \label{eq:error:estimate}
   \norm{u - \scsol}{} \;\rev{\le}\;
   \rev{\big(1 - q_{\rm sat}\big)^{-1} \norm{\scsolhat - \scsol}{}}
   \le \big(1 - q_{\rm sat}\big)^{-1} \big(\mu_\bullet + \tau_\bullet\big).
\end{equation}
\rev{Here, $\mu_\bullet$ and $\tau_\bullet$ are, respectively, the spatial and parametric error estimates defined as follows
(cf.~\cite[eq.~(24)]{Bespalov22} and~\cite[\S4.2 and Remarks~1 and~4]{Bespalov22}, respectively):
\begin{align}  \label{eq:spatial:estimate}
   \mu_\bullet   & :=  \norm{S_\bullet (\widehat U_\bullet - U_\bullet)}{} =
    \norm[\bigg]{\sum\limits_{\z \in \Colpts_\bullet} (\widehat u_{\bullet \z} - u_{\bullet \z})\, \LagrBasis{\bullet \z}{}}{}
\end{align}
and
\begin{equation} \label{eq:param:estimate}
   \tau_\bullet  := \norm[\bigg]{ \sum\limits_{\nnu \in \rmarg_\bullet}
                                                 \Delta^{\bmf(\nnu)}
                                                  \sum\limits_{\z \in \Colpts_{\indset_\bullet \cup \rmarg_\bullet}}
                                                           u_{0 \z} 
                                                           \LagrBasisHat{\bullet \z}{}
                                                  }{},
\end{equation}
where
$\{\LagrBasisHat{\bullet \z}{}(\y), \z \in \Colpts_{\indset_\bullet \cup \rmarg_\bullet} \}$ is a set of
\rev{multivariate} Lagrange basis functions constructed for the set of collocation points $\Colpts_{\indset_\bullet \cup \rmarg_\bullet}$ and satisfying
$\LagrBasisHat{\bullet \z}{}(\z') = \delta_{\z\z'}$ for any $\z,\, \z' \in \Colpts_{\indset_\bullet \cup \rmarg_\bullet}$.}

\rev{
\begin{remark} \label{remark:saturation}
While the saturation assumption can be empirically justified when numerical approximations exhibit some asymptotic behavior,
the rigorous proofs exist either in the context of deterministic problems with constant coefficients
(see, e.g.,~\cite{dn2002,cgg2016}) or in the context of stochastic Galerkin FEM for parametric PDEs (see~\cite{BachmayrEEV_CAF}).
When required for generic finite element approximations, the saturation assumption may fail
(see~\cite{bek96} for a counterexample in the deterministic setting).
However, while our main result in the present paper (Theorem~\ref{theorem:main}) is proved independently of the saturation assumption~\eqref{eq:saturation},
for the convergence result in Corollary~\ref{cor:main} the saturation assumption is required only for a sequence of
SC-FEM approximations living in nested discrete subspaces $\big( \P_{\indset_\ell}(\G;\X_\ell) \big)_{\ell \in \N_0}$
generated by the adaptive algorithm.
\end{remark}
}

Let us now turn to the associated \rev{spatial and parametric} error indicators  (see~\rev{\cite[\S4.1 and \S4.2]{Bespalov22}}).
For each collocation point $\z \in \Colpts_\bullet$, one can first compute local \rev{(spatial)} two-level error indicators
associated with \rev{new interior vertices created by uniform refinement of $\TT_\bullet$}
(recall that the same mesh $\TT_\bullet$ is assigned to each collocation point):
\begin{equation} \label{eq:2level:local:indicator}
   \mu_{\bullet \z}(\xi) :=
   \frac{\big| (f,\widehat\varphi_{\bullet,\xi})_{L^2(D)} - (a(\cdot,\z) \nabla u_{\bullet \z}, \nabla \widehat\varphi_{\bullet,\xi})_{L^2(D)} \big|}
           {\norm{\widehat\varphi_{\bullet,\xi}}{\X}}\quad
   \text{for all } \xi \in \NN_{\bullet}^+.
\end{equation}
These indicators are then combined to produce the \emph{spatial} error indicator for each $\z \in \Colpts_\bullet$:
\begin{equation} \label{eq:2level:indicator}
   \mu^2_{\bullet \z} := \sum\limits_{\xi \in \NN_{\bullet}^+} \mu^2_{\bullet \z}(\xi).
\end{equation}
Local mesh refinement is effected by using D{\"o}rfler marking on local error indicators~\eqref{eq:2level:local:indicator}
to find a set of marked vertices $\MM_\bullet \subseteq \NN_\bullet^+$.
Then the mesh $\TT_{\circ} := \refine(\TT_\bullet,\MM_\bullet)$ is the refinement of $\TT_\bullet$ such that
$\MM_\bullet \subset \NN_{\circ}$, i.e., all marked vertices of $\widehat\TT_{\bullet}$ are vertices of $\TT_{\circ}$.

\rev{In order to introduce parametric error indicators, we exploit a}
useful property of the reduced margin that for a monotone $\indset_\bullet$ and for any subset of marked indices 
$\text{M}_\bullet \subseteq \rmarg_\bullet$, the index set $\indset_\bullet \cup \text{M}_\bullet$ is also monotone.
Therefore, for each index $\nnu \in \rmarg_\bullet$, a natural parametric error indicator is given by
the norm of the hierarchical surplus associated with the parametric enhancement
as a result of adding $\nnu$ to $\indset_\bullet$:
\begin{equation} \label{eq:param:indicator}
    \tau_{\bullet  \nnu} = \tau_{\bullet  \nnu}[\ssolc] :=
                                       \norm[\bigg]{ \Delta^{\bmf(\nnu)} \sum\limits_{\z \in \Colpts_{\indset_\bullet \cup \rmarg_\bullet}}
                                                                       u_{0\z} \LagrBasisHat{\bullet \z}{}}{}
                                    = \norm[\bigg]{ \Delta^{\bmf(\nnu)}
                                                              \sum\limits_{\mmu \in \indset_\bullet \cup \rmarg_\bullet}
                                                                       \Delta^{\bmf(\mmu)} \ssolc}{},
\end{equation}
where $\ssolc$ is the semidiscrete approximation satisfying~\eqref{eq:semidiscrete} with
$\W$ replaced by $\X_0$.

In addition to $\tau_{\bullet  \nnu}[\ssolc]$, we introduce two other parametric indicators for each $\nnu \,{\in}\, \rmarg_\bullet$:
$\tau_{\bullet  \nnu}[\ssol]$ and $\tau_{\bullet  \nnu}[\ssolh]$.
Here, $\ssol$ (resp., $\ssolh$) is the semidiscrete approximation satisfying~\eqref{eq:semidiscrete} with
$\W$ replaced by $\X_\bullet$ (resp., $\widehat\X_\bullet$).
Specifically, for $w \in \left\{\ssol, \ssolh \right\}$, we define
\begin{equation} \label{eq:param:indicator1} 
    \tau_{\bullet \nnu}[w] := \norm[\bigg]{ \Delta^{\bmf(\nnu)} \sum\limits_{\z \in \Colpts_{\indset_\bullet \cup \rmarg_\bullet}}
                                                                    w_{\bullet \z} \LagrBasisHat{\bullet \z}{}}{}
                                        = \norm[\bigg]{ \Delta^{\bmf(\nnu)} \sum\limits_{\mmu \in \indset_\bullet \cup \rmarg_\bullet}
                                                                    \Delta^{\bmf(\mmu)} w}{},
\end{equation}
where
\[
  w_{\bullet \z} = 
  \begin{cases}
  u_{\bullet\z} & \text{if $w = \ssol$},\\
  \widehat u_{\bullet\z} & \text{if $w = \ssolh$}.
  \end{cases}
\]

\begin{remark} \label{remark:param:indicators}
We emphasize that
in order to steer the adaptive refinement in Algorithm~{\rm \ref{algorithm}} below,
we only use the parametric error indicators $\tau_{\bullet  \nnu} = \tau_{\bullet  \nnu}[\ssolc]$
that are associated with approximations on the coarsest mesh and hence cheap to compute.
The two other parametric indicators, $\tau_{\bullet  \nnu}[\ssol]$ and $\tau_{\bullet  \nnu}[\ssolh]$,
are significantly more expensive to compute.
While these indicators are not part of the adaptive algorithm, they arise as theoretical tools in our analysis
(see Theorem~{\rm \ref{theorem:parametic:alt}} and its application in the proof of the main result in Theorem~{\rm \ref{theorem:main}}).
\end{remark}

We refer to~\cite[section~3]{BS23} for a discussion of computational costs
associated with computing the error estimates $\mu_\bullet$ and $\tau_\bullet$.
The key point is that in practice, the computation of these error estimates is only required to give
a reliable criterion for termination of the adaptive process and, therefore, can be done periodically.
On the other hand, the error indicators $\mu_{\bullet\z}$ and $\tau_{\bullet\nnu}$ are cheaper to compute
and the following inequalities hold (see equations~(31)--(34) and Remark~3 in~\cite{Bespalov22}):
\begin{equation} \label{eq:err:indicators}
   \mu_\bullet \lesssim \sum\limits_{\z \in \Colpts_\bullet} \mu_{\bullet\z}\, \|L_{\bullet\z}\|_{L^p_\pi(\G)}
   \quad\text{and}\quad
   \tau_\bullet \le \sum\limits_{\nnu \in \rmarg_\bullet} \tau_{\bullet\nnu}.
\end{equation}
This motivates the use of the
error indicators in the marking strategy within the adaptive algorithm.

\begin{algorithm} \label{algorithm}
{\bfseries Input:}
$\indset_0 = \{ \1 \}$; $\TT_{0}$.\\
Set the iteration counter $\ell := 0$.
\begin{itemize}
\item[\rm(i)] 

Compute Galerkin approximations $\big\{ u_{\ell \z} \in \X_{\ell}: \z \in \Colpts_{\indset_\ell} \big\}$ 
by solving~\eqref{eq:sample:fem}.
\item[\rm(ii)] 
Compute the spatial error indicators $\big\{ \mu_{\ell\z}: \z \in \Colpts_\ell \big\}$ 
given by~\eqref{eq:2level:indicator}.
\item[\rm(iii)] 

Compute Galerkin approximations $\big\{ u_{0 \z} \in \X_0 : \z \in \Colpts_{\indset_\ell \cup \rmarg_\ell} \setminus \Colpts_{\indset_\ell} \big\}$ 
by solving~\eqref{eq:sample:fem}.
\item[\rm(iv)] 

Compute the parametric error indicators
$\big\{ \tau_{\ell \nnu} : \nnu \in \rmarg_\ell \big\}$
given by~\eqref{eq:param:indicator}.
\item[\rm(v)] 
Use the marking criterion in Algorithm{\rm ~\ref{algorithm_m}}
to determine $\MM_{\ell \z} \subseteq \NN_{\ell}^+$ for all $\z \in \Colpts_\ell$, 
$\markindset_\ell \subseteq \rmarg_\ell$ and, if $\markindset_\ell \not= \emptyset$,
$\nnu^*_\ell \in  \rmarg_\ell \setminus \markindset_\ell$.
\item[\rm(vi)]Set $\TT_{\ell+1} := \refine(\TT_{\ell}, \bigcup_{\z \in \Colpts_\ell}\MM_{\ell \z})$, and 
$\indset_{\ell+1} := \indset_\ell \cup \markindset_\ell \cup \{\nnu^*_\ell\}$.
\item[\rm(vii)] Increase the counter $\ell \mapsto \ell+1$ and goto {\rm(i)}.
\end{itemize}
{\bfseries Output:}
$\big( \TT_\ell,\, \indset_\ell,\, u_{\ell}^{\rm SC},\, \mu_\ell + \tau_\ell \big)_{\ell \in \N_0}$,
where the SC-FEM approximation $u_{\ell}^{\rm SC}$ is computed via~\eqref{eq:scfem:sol} from Galerkin approximations 
$\big\{ u_{{\ell} \z} \in \X_{\ell} : \z \in \Colpts_{\ell} \big\}$ and the error estimates $\mu_{\ell}$ and $\tau_{\ell}$ are given
by~\eqref{eq:spatial:estimate} and~\eqref{eq:param:estimate}, respectively.
\end{algorithm}
The following D{\"o}rfler-type marking strategy is used for step~(v) of Algorithm~\ref{algorithm}.

\begin{algorithm} \label{algorithm_m}
{\bfseries Input:}
error indicators $\big\{ \mu_{\ell \z}(\xi) : \z \in \Colpts_\ell,\ \xi \in \NN_{\ell}^+ \big\}$,
$\big\{ \mu_{\ell \z} : \z \in \Colpts_\ell \big\}$,
and
$\big\{ \tau_{\ell \nnu} : \nnu \in \rmarg_\ell\big\}$;
marking parameters $0 < \theta_\X, \theta_\Colpts \le 1$ and $\vartheta > 0$.
\begin{itemize}
\item[$\bullet$]
If \
$\sum_{\z \in \Colpts_\ell} \mu_{\ell \z} \norm{L_{\ell \z}}{L^p_{\pi}(\G)}
  \ge \vartheta \sum_{\nnu \in \rmarg_\ell} \tau_{\ell \nnu}$,
then proceed as follows (spatial refinement):
\begin{itemize}
\item[$\circ$]
set $\markindset_\ell := \emptyset$;
\item[$\circ$]
for each $\z \in \Colpts_\ell$,
determine $\MM_{\ell \z} \subseteq \NN_{\ell}^+$ of minimal cardinality such that
\begin{equation} \label{eq:doerfler:separate1}
 \theta_\X \, \mu_{\ell \z}^2 
 \le \sum_{\xi \in \MM_{\ell \z}} \mu^2_{\ell \z}(\xi).
\end{equation}
\end{itemize}
\item[$\bullet$]
Otherwise,
proceed as follows (parametric enrichment):
\begin{itemize}
\item[$\circ$]
set $\MM_{\ell \z} := \emptyset$ for all $\z \in \Colpts_\ell$;
\item[$\circ$]
determine the set $\markindset_\ell \subseteq \rmarg_\ell$ of minimal cardinality such that
\begin{subequations} \label{eq:param:marking}
\begin{equation} \label{ineq:dorfler}
 \theta_\Colpts \, \sum_{\nnu \in \rmarg_\ell} \tau_{\ell \nnu} \le
 \sum_{\nnu \in \markindset_{\ell}} \tau_{\ell \nnu};
\end{equation}

\item[$\circ$] 
determine $\nnu^*_{\ell} \in \rmarg_\ell \setminus  \markindset_{\ell}$ such that
\begin{equation} \label{eq:index:min}
    \nnu^*_{\ell} = \arg \min\limits_{\nnu \in \rmarg_\ell \setminus \markindset_\ell} \|\nnu\|_1;
\end{equation}
if there are several $\nnu^*_\ell \in \rmarg_\ell \setminus \markindset_\ell$ satisfying~\eqref{eq:index:min},
then choose the one that comes first in lexicographic ordering.

\end{subequations}
\end{itemize}
\end{itemize}
{\bfseries Output:}
$\MM_{\ell \z} \subseteq \NN_{\ell}^+$ for all $\z \in \Colpts_\ell$,
$\markindset_\ell \subseteq \rmarg_\ell$ and, if $\markindset_\ell \not= \emptyset$,
$\nnu^*_{\ell} \in \rmarg_\ell \setminus  \markindset_{\ell}$.
\end{algorithm}

\begin{remark} \label{remark:param:marking}
In the above marking strategy, parametric D{\"o}rfler marking is complemented by adding a multi-index
$\nnu^*_\ell \in \rmarg_\ell \setminus  \markindset_{\ell}$ of the smallest magnitude (in the sense of the $1$-norm).
In practice, using only the D{\"o}rfler marking criterion given by~\eqref{ineq:dorfler} tends to be sufficient for the adaptive algorithm
to generate converging SC-FEM approximations for representative test problems
(see~\cite[section~5]{Bespalov22}). 
However, in a general case of the parametric elliptic PDE given by~\eqref{eq:pde:strong},
adding a multi-index $\nnu^*_\ell$ satisfying~\eqref{eq:index:min} is required in our analysis
to guarantee convergence of adaptive SC-FEM~approximations.
\end{remark}

\section{Convergence of parametric error estimates} \label{sec:param:convergence}

The goal of this section is to show that 
$\lim\limits_{k \to \infty} \tau_{\ell_k} = \lim\limits_{k \to \infty} \sum\limits_{\nnu \in \rmarg_{\ell_k}} \tau_{\ell_k  \nnu} = 0$
along the subsequence $\left(\ell_k\right)_{k \in \N_0}$ of iterations where parametric enrichments occur in Algorithm~\ref{algorithm}.
We follow the idea that was used in~\cite{eest22} in order to prove convergence
of the adaptive algorithm proposed in~\cite{GerstnerGriebel2003}.
We start by collecting some auxiliary results.
The following lemma establishes a useful property of hierarchical surplus operators.

\begin{lemma}[{\cite[Theorem~2.3]{FeischlS21}}] \label{lemma:interpolators}
    Let $\nnu, \mmu \in \N^M$ be two multi-indices such that $\nu_m<\mu_m$ for some $m \in \{1,2,\ldots,M\}$.
    Then
    $\Delta^{\bmf(\nnu)} \Delta^{\bmf(\mmu)} v(\y) \equiv 0$
    for any $v \in C^0(\Gamma; \X)$.    
\end{lemma}

Next, we formulate the following abstract result for $l^p$-sequences.
This result was originally proved for $p=2$ in~\cite[Lemma~15]{bprr18}.
However, the proof is easy to generalize to the case of arbitrary $p \in [1, \infty)$; cf.~\cite[Lemma~2.5]{eest22}.

\begin{lemma}[{\cite[Lemma~15]{bprr18}}] \label{lemma:parametric}
Let $(x_n)_{n\in\N} \subset \R_{\ge0}$ and $(x_n^{(\ell)})_{n\in\N} \subset \R_{\ge0}$ with $\ell \in \N_0$.
Let $p \in [1, \infty)$ and assume that $(x_n)_{n\in\N} \in l^p(\N)$ and
$\norm{x_n - x_n^{(\ell)}}{l^p(\N)} \to 0$ as $\ell \to \infty$.
In addition, 
let $g : \R_{\ge0} \to \R_{\ge0}$ be a continuous function with $g(0) = 0$ and
assume that there exists a sequence $(\PP_\ell)_{\ell \in \N_0}$ of nested subsets of $\N$
(i.e., $\PP_\ell \subseteq \PP_{\ell+1}$ for all $\ell \in \N_0$) satisfying the following property:
\begin{equation} \label{eq:lemma:parametric}
   x_m^{(\ell)} \le g \Bigg( \sum_{n \in \PP_{\ell+1} \setminus \PP_\ell} (x_n^{(\ell)})^p \Bigg)\quad
   \text{for all } \ell \in \N_0 \text{ and } m \in \N \setminus \PP_{\ell+1}.
\end{equation}
Then
$\sum_{n \in \N \setminus \PP_\ell} x_n^p \to 0$ as $\ell \to \infty$.
\end{lemma}

Now, let $\indset_{\infty} := \cup_{\ell \in \N_0} \indset_{\ell}$ and $\rmarg_\infty := \rmarg(\indset_\infty)$.
For each $\ell \in \N_0 \cup \{\infty\}$, let us consider the following sequence:
\begin{equation} \label{eq:sequences}
    \widehat\tau_\ell := \big(\widehat\tau_{\ell \nnu}\big)_{\nnu \in \N^M}\quad \text{with \ }
    \widehat\tau_{\ell\nnu} = 
     \begin{cases}
        \tau_{\ell\nnu},
        & \nnu \in \indset_{\ell} \cup \rmarg_{\ell},\\
        0,
        & \nnu \in \N^M \setminus (\indset_{\ell} \cup \rmarg_{\ell}),
    \end{cases}
\end{equation}
where $\tau_{\ell\nnu}$ are defined according to~\eqref{eq:param:indicator}
for $\nnu \in \rmarg_\ell$ as well as for $\nnu \in \indset_\ell$.

\begin{lemma} \label{lemma:sequences}
For any $\ell \in \N_0$, the sequence $\widehat\tau_\ell$ is a subsequence of $\widehat\tau_\infty$.
\end{lemma}

\begin{proof}
Let $\ell \in \N_0$ and consider a multi-index
$\nnu \in \big(\indset_\ell \cup \rmarg_\ell\big) \subset \big(\indset_\infty \cup \rmarg_\infty \big)$.
Using the definition of the reduced margin in~\eqref{eq:reduced:margin} and applying Lemma~\ref{lemma:interpolators},
we obtain
\begin{align*}
   \tau_{\infty\nnu} \reff{eq:param:indicator}=
   \norm[\bigg]{\Delta^{\bmf(\nnu)} \sum\limits_{\mmu \in \indset_\infty \cup \rmarg_\infty} \Delta^{\bmf(\mmu)} \ssolc}{}
   & =
   \norm[\bigg]{\Delta^{\bmf(\nnu)}
   \bigg( \sum\limits_{\mmu \in \indset_\ell \cup \rmarg_\ell} +
             \sum\limits_{\mmu \in (\indset_\infty \cup \rmarg_\infty) \setminus (\indset_\ell \cup \rmarg_\ell)} \bigg)
    \Delta^{\bmf(\mmu)} \ssolc}{}\\[4pt]
   & =
   \norm[\bigg]{\Delta^{\bmf(\nnu)} \sum\limits_{\mmu \in \indset_\ell \cup \rmarg_\ell} \Delta^{\bmf(\mmu)} \ssolc}{}
   \reff{eq:param:indicator} = \tau_{\ell\nnu}.
\end{align*}
\noindent
This proves the statement of the lemma.
\end{proof}

Now, we are ready to establish convergence of parametric error estimates
along the subsequence of iterations for which parametric enrichment takes place.

\begin{theorem} \label{theorem:parametric}
Suppose that the Taylor coefficients $[\ssolc]_{\ii}$, $\ii \in \N_0^M$, defined by~\eqref{eq:Taylor:coeff}
with $w =\ssolc$ satisfy the summability property~\eqref{eq:summability:property}. 
Let $(\ell_k)_{k \in \N_0} \subset \N_0$ denote the subsequence of iterations
where parametric enrichment occurs in Algorithm~{\rm \ref{algorithm}}
and assume that $\ell_k\, {\xrightarrow{k \to \infty}}\, \infty$.
Then
\rev{the subsequences 
\[
   \bigg( \sum\limits_{\nnu \in \rmarg_{\ell_k}} \tau_{\ell_k\nnu} \bigg)_{k \in \N_0}
   \quad\text{and}\quad
   \big( \tau_{\ell_k} \big)_{k \in \N_0}
\]
converge to zero.}
\end{theorem}

\begin{proof}

We omit the subscript $k$ to simplify notation in the proof and assume that $\ell = \ell_k {\xrightarrow{k \to \infty}}\, \infty$.
Using the sequences introduced in~\eqref{eq:sequences} and the triangle inequality, we estimate

\begin{equation} \label{eq:err:sequence}
    \tau_{\ell} \reff{eq:err:indicators}{\le}
    \sum\limits_{\nnu \in \rmarg_\ell} \tau_{\ell\nnu} \le
    \sum\limits_{\nnu \in \rmarg_\ell} \widehat\tau_{\infty\nnu} +
    \sum\limits_{\nnu \in \rmarg_\ell} | \widehat\tau_{\ell\nnu} - \widehat\tau_{\infty\nnu} | \le
    \sum\limits_{\nnu \in \rmarg_\ell} \widehat\tau_{\infty\nnu} + \norm[\big]{\widehat\tau_{\infty} - \widehat\tau_{\ell}}{l^1(\N^M)}.
\end{equation}
We will complete the proof by showing that each term on the right-hand side of~\eqref{eq:err:sequence} converges to zero as $\ell \to \infty$.
We will do this in three steps.

{\bf Step~1.}
First, we show that $\widehat\tau_\ell \in l^1(\N^M)$ for any $\ell \in \N_0 \cup \{\infty\}$.
Let $\ell \in \N_0$.
For any $\nnu \in \indset_\ell \cup \rmarg_\ell$ we have
\begin{align*} 
    \tau_{\ell \nnu}
    & \reff{eq:param:indicator}=
    \norm[\bigg]{\Delta^{\bmf(\nnu)} \sum\limits_{\mmu \in \indset_\ell \cup \rmarg_\ell} \Delta^{\bmf(\mmu)} \ssolc}{}
    \\[4pt]
    & \refp{eq:param:indicator}=
    \norm[\bigg]{\Delta^{\bmf(\nnu)}
                         \bigg( \sum\limits_{\mmu \in \indset_\ell \cup \rmarg_\ell} +
                                    \sum\limits_{\mmu \in \N^M \setminus (\indset_\ell \cup \rmarg_\ell)}
                          \bigg)
                         \Delta^{\bmf(\mmu)} \ssolc}{}  =
    \norm[\bigg]{\Delta^{\bmf(\nnu)} \sum\limits_{\mmu \in \N^M} \Delta^{\bmf(\mmu)} \ssolc}{}
    \\[4pt]
    & \reff{eq:poly:representation}=
    \norm[\bigg]{\Delta^{\bmf(\nnu)} \sum\limits_{\ii \in \N_0^M} [\ssolc]_{\ii} P_{\ii}}{} =
    \norm[\bigg]{\sum\limits_{\ii \in \N_0^M} [\ssolc]_{\ii} \, \Delta^{\bmf(\nnu)} P_{\ii} }{},
\end{align*}
where we used Lemma~\ref{lemma:interpolators} in the second equality.
Hence, applying the triangle inequality and then the Cauchy-Schwarz inequality, we obtain
\begin{align} \label{eq:rho-estimate1}
    \tau_{\ell \nnu}
    \le
    \bigg( \sum\limits_{\ii \in \N_0^M} \rro^{2\ii} \normm[\big]{[\ssolc]_{\ii}}{\X} \bigg)^{1/2}\,
    \bigg( \sum\limits_{\ii \in \N_0^M} \rro^{-2\ii} \normm[\big]{\Delta^{\bmf(\nnu)} P_{\ii}}{L_\pi^p(\G)} \bigg)^{1/2}.
\end{align}
The summability property~\eqref{eq:summability:property} for the Taylor coefficients $[\ssolc]_{\ii}$ implies
\begin{equation} \label{eq:const:1}
   \bigg( \sum\limits_{\ii \in \N_0^M} \rro^{2\ii} \normm[\big]{[\ssolc]_{\ii}}{\X} \bigg)^{1/2} =: C_1 < \infty. 
\end{equation}  
Now, let us consider the second factor on the right-hand side of~\eqref{eq:rho-estimate1}.
Firstly, introducing the Lebesgue constant $\leb(\nnu)$ of the hierarchical surplus operator $\Delta^{\bmf(\nnu)}$
(with respect to the $L_\pi^{\infty}(\G)$-norm) and using the fact that $\G = [-1, 1]^M$, we estimate
\begin{align} \label{eq:L^p-norm:estimate}
    \norm[\big]{\Delta^{\bmf(\nnu)} P_{\ii}}{L_\pi^{p}(\G)} 
    \le
    \norm[\big]{\Delta^{\bmf(\nnu)} P_{\ii}}{L_\pi^{\infty}(\G)}
    \le
    \leb(\nnu) \norm[\big]{ P_{\ii}}{L_\pi^{\infty}(\G)} =
    \leb(\nnu) \cdot \max\limits_{\y \in \G} |\y^{\ii}| =
    \leb(\nnu).
\end{align}
Secondly, since $\Lagr_m^{\mf(\nu_m)} g = g$ for any (univariate) polynomial $g$
of degree $\le \mf(\nu_m)-1$, we find~that
\begin{equation*} 
    \Delta^{\bmf(\nnu)} P_{\ii}(\y) = \prod_{m=1}^M \Delta_m^{\mf(\nu_m)} P_{i_m}(y_m) =
    \prod_{m=1}^M \Big( \Lagr_m^{\mf(\nu_m)} - \Lagr_m^{\mf(\nu_m-1)} \Big) P_{i_m}(y_m) \equiv 0,
\end{equation*}
provided that $i_m \le \max\,\{0, \mf(\nu_m-1) - 1\}$ for at least one $m \in \{1,\ldots,M\}$.
Consequently,  $\Delta^{\bmf(\nnu)} P_{\ii}(\y) \not\equiv 0 $ for
$\ii \ge \bmf(\nnu - \boldsymbol{1}) \ge \nnu - \boldsymbol{1}$.
Thus, the second sum on the right-hand side of~\eqref{eq:rho-estimate1} can be estimated as follows:
\begin{align} \label{eq:rho-estimate2}
    \sum\limits_{\ii \in \N_0^M} \rro^{-2\ii} \normm[\big]{\Delta^{\bmf(\nnu)} P_{\ii}}{L_\pi^p(\G)}
    & \refp{eq:L^p-norm:estimate}=
    \sum\limits_{\ii \ge \nnu - \boldsymbol{1}} \rro^{-2\ii} \normm[\big]{\Delta^{\bmf(\nnu)} P_{\ii}}{L_\pi^{p}(\G)}
    \nonumber
    \\[4pt]
    & \reff{eq:L^p-norm:estimate}\le
    \leb^2(\nnu) \sum\limits_{\ii \ge \nnu - \boldsymbol{1}} \rro^{-2\ii} =
    \leb^2(\nnu) \, C^2_{\rro} \, \rro^{-2\nnu},
\end{align}
where at the last step we used a finite product of geometric series to calculate the infinite sum over multi-indices as follows:
\begin{equation*} 
     \sum\limits_{\ii \ge \nnu - \boldsymbol{1}} \rro^{-2\ii} =
     \prod_{m=1}^M \sum\limits_{i_m \ge \nu_m - 1} \ro_m^{-2i_m}
     = \prod_{m=1}^M \frac{\ro_m^{-2(\nu_m-1)}}{1-\ro_m^{-2}}
     = C^2_{\rro} \, \rro^{-2\nnu}\ \ \text{with $C^2_{\rro} := \prod_{m=1}^M \frac{\ro_m^2}{1-\ro_m^{-2}}$}.
\end{equation*}
Now, combining~\eqref{eq:rho-estimate1},~\eqref{eq:const:1}, and~\eqref{eq:rho-estimate2}, we~obtain
\begin{equation} \label{eq:tau:est:final}
    \tau_{\ell \nnu} \le
    C_1 \, C_{\rro} \, \leb(\nnu) \, \rro^{-\nnu} \lesssim  \leb(\nnu) \, \rro^{-\nnu}\quad
    \forall\,\nnu \in \indset_\ell \cup \rmarg_\ell.
\end{equation}
The Lebesgue constant of the hierarchical surplus operator can be estimated as follows:
\begin{equation} \label{eq:lebesgue}
    \leb(\nnu) \lesssim
    \begin{cases}
        \prod_{m=1}^M \nu_m, &\text{for CC points,}
        \\[4pt]
        \prod_{m=1}^M \nu^2_m \log{\nu_m}, & \text{for Leja points};
    \end{cases}
\end{equation}
for CC 
points, we refer to~\cite[Section~2.1]{FeischlS21} and
for Leja points the bound follows from~\cite[Theorem~3.1]{Chk13} using the idea in~\cite[Section~2.1]{FeischlS21}.
Hence, $\left(\leb(\nnu) \rro^{-\nnu} \right)_{\nnu \in \N^M} \in l^1(\N^M)$
due to the integral convergence test for series.
Therefore, recalling the definition of $\widehat\tau_{\ell\nnu}$ in~\eqref{eq:sequences},
we conclude from~\eqref{eq:tau:est:final} that $\widehat\tau_\ell \in l^1(\N^M)$ for any $\ell \in \N_0$.
Due to Lemma~\ref{lemma:sequences}, each element of $\widehat\tau_\infty$ can be bounded in the same way, i.e.,
$\tau_{\infty \nnu} \lesssim  \leb(\nnu) \, \rro^{-\nnu}$ for all $\nnu \in \indset_\infty \cup \rmarg_\infty$.
Since $\left(\leb(\nnu) \rro^{-\nnu} \right)_{\nnu \in \N^M} \in l^1(\N^M)$, recalling the definition of $\widehat\tau_{\infty\nnu}$ in~\eqref{eq:sequences}, we conclude that $\widehat\tau_\infty \in l^1(\N^M)$.

{\bf Step~2.}
Next, we prove that
$\norm[\big]{\widehat\tau_{\infty} - \widehat\tau_{\ell}}{l^1(\N^M)} \to 0$ as $\ell \to \infty$.
We again apply Lemma~\ref{lemma:sequences} to derive
\begin{align*}
    \norm[\big]{\widehat\tau_{\infty} - \widehat\tau_{\ell}}{l^1(\N^M)}
    & =
    \sum\limits_{\nnu \in (\indset_\infty \cup \rmarg_\infty)\setminus (\indset_\ell \cup \rmarg_\ell)} \tau_{\infty\nnu} =
    \sum\limits_{n=\ell+1}^{\infty} \;
    \sum\limits_{\nnu \in (\indset_n \cup \rmarg_n)\setminus (\indset_{n-1} \cup \rmarg_{n-1})} \tau_{\infty\nnu}
    \nonumber
    \\[4pt]
    & =
    \sum\limits_{n=\ell+1}^{\infty} \;
    \sum\limits_{ \nnu \in (\indset_n \cup \rmarg_n)\setminus (\indset_{n-1} \cup \rmarg_{n-1})} \tau_{n\nnu}
    \reff{eq:tau:est:final}\lesssim
    \sum\limits_{n=\ell+1}^{\infty} \;
    \sum\limits_{ \nnu \in (\indset_n \cup \rmarg_n)\setminus (\indset_{n-1} \cup \rmarg_{n-1})} \leb(\nnu) \rro^{-\nnu}
    \nonumber
    \\[4pt]
    & =
    \sum\limits_{ \nnu \in (\indset_\infty \cup \rmarg_\infty)\setminus (\indset_\ell \cup \rmarg_\ell)} \leb(\nnu) \rro^{-\nnu}.  
\end{align*}
Since $\left(\leb(\nnu) \rro^{-\nnu} \right)_{\nnu \in \N^M} \in l^1(\N^M)$ and
$\indset_\infty \cup \rmarg_\infty = \cup_{\ell \in \N_0} \indset_\ell \cup \rmarg_\ell$,
we conclude that
\begin{equation} \label{eq:limit:sequence}
    \lim\limits_{\ell \to \infty}  \norm[\big]{\widehat\tau_{\infty} - \widehat\tau_{\ell}}{l^1(\N^M)}  = \lim\limits_{\ell \to \infty} \sum\limits_{ \nnu \in (\indset_\infty \cup \rmarg_\infty)\setminus (\indset_\ell \cup \rmarg_\ell)} \leb(\nnu) \rro^{-\nnu} = 0.
\end{equation}

{\bf Step~3.}
In this step, we apply Lemma~\ref{lemma:parametric} with $p=1$ to prove that
$\sum_{\nnu \in \rmarg_\ell} \widehat\tau_{\infty\nnu} \to 0$ as $\ell \to \infty$.
In fact, all hypotheses of Lemma~\ref{lemma:parametric} are
satisfied in the present setting:
\begin{itemize}
    \item[$\bullet$]
    $\N^M$ is a countable set that can be identified (via a one-to-one map) with $\N$;

    \item[$\bullet$]
    $\left( \widehat\tau_{\infty\nnu}\right)_{\nnu \in \N^M} \in l^1(\N^M)$ is identified with a sequence
    $\left(x_n\right)_{n \in \N} \in l^1(\N)$;

    \item[$\bullet$]
    
    $\left( \widehat\tau_{\ell\nnu}\right)_{\nnu \in \N^M}$
    is identified with a sequence
    $\big(x^{(\ell)}_n\big)_{n \in \N}$
    for all $\ell \in \N_0$;

    \item[$\bullet$]
    after this identification, we conclude that
    $\lim\limits_{\ell \to \infty}  \norm[\big]{x_n  - x_n^{(\ell)}}{l^1(\N)}  \reff{eq:limit:sequence}= 0$;

    \item[$\bullet$]
    $\indset_\ell \subset \N^M$ is
    identified with a set $\PP_\ell \subset \N$
    for each $\ell \in \N_0$;
    
    \item[$\bullet$]
    the set of newly added indices $\indset_{\ell+1} \setminus \indset_\ell = \markindset_\ell \cup \{\nnu^*_\ell\}$ is thus
    identified with $\PP_{\ell+1} \setminus \PP_\ell$;

    \item[$\bullet$]
    with this identification, one has $\PP_\ell \subset \PP_{\ell+1}$ for $\ell \in \N_0$, and
    D{\"o}rfler marking~\eqref{ineq:dorfler} with $0 < \theta_\Colpts \le 1$ implies inequality~\eqref{eq:lemma:parametric}
    with $g(s) := \frac{1-\theta_\Colpts}{\theta_\Colpts} s $; indeed,
    \begin{equation*}
        \theta_\Colpts
        \bigg(
        \sum_{\nnu \in \rmarg_\ell\setminus \left(\markindset_\ell \cup \{\nnu^*_\ell\}\right)} \tau_{\ell \nnu} + 
        \sum_{\nnu \in \markindset_\ell \cup \{\nnu^*_\ell\}} \tau_{\ell\nnu}
        \bigg)
        \stackrel{\eqref{ineq:dorfler}}\le
        \sum_{\nnu \in \markindset_{\ell}} \tau_{\ell\nnu} \le
        \sum_{\nnu \in \markindset_\ell \cup \{\nnu^*_\ell\}} \tau_{\ell\nnu}
    \end{equation*}
    and therefore
    \begin{equation*} 
            \widehat\tau_{\ell\nnu} \le
            \sum_{\nnu \in \rmarg_\ell\setminus \left(\markindset_\ell \cup \{\nnu^*_\ell\}\right)} \tau_{\ell \nnu} \le
            \frac{1-\theta_\Colpts}{\theta_\Colpts}
            \sum_{\mmu \in \markindset_\ell \cup \{\nnu^*_\ell\}} \widehat\tau_{\ell\mmu}\quad
            \forall \nnu \in \N^M \setminus \indset_{\ell+1}.          
    \end{equation*}
\end{itemize}
Hence, applying Lemma~\ref{lemma:parametric}, we prove that
$\sum\limits_{\nnu \in \N^M \setminus \indset_\ell} \widehat\tau_{\infty\nnu} \xrightarrow{\ell \to \infty} 0$.
Therefore, recalling that $\rmarg_\ell \subset \N^M \setminus \indset_\ell$, we conclude that
$\sum\limits_{\nnu \in \rmarg_\ell} \widehat\tau_{\infty\nnu} \xrightarrow{\ell \to \infty} 0$.

Now, the proof is completed by applying the results of Steps~2 and~3 to the right-hand side of~\eqref{eq:err:sequence} and
recalling that $\ell = \ell_k \xrightarrow{k \to \infty} \infty$.
\end{proof}

Next, we prove convergence of
\rev{two subsequences associated with alternative parametric error indicators given by}~\eqref{eq:param:indicator1} with
$w \in \left\{\ssol, \ssolh \right\}$.

\begin{theorem} \label{theorem:parametic:alt}
Suppose that the diffusion coefficient $a(x,\y)$ satisfies the assumptions of
either Lemma~{\rm \ref{lemma:summability:affine}} or Lemma~{\rm \ref{lemma:summability:general}}.
Let $(\ell_k)_{k \in \N_0} \subset \N_0$ denote the subsequence of iterations
where parametric enrichment occurs in Algorithm~{\rm \ref{algorithm}}
and assume that $\ell_k\, {\xrightarrow{k \to \infty}}\, \infty$.
Then
\rev{the subsequences
\[
   \bigg( \sum\limits_{\nnu \in \rmarg_{\ell_k}} \tau_{\ell_k\nnu}[\ssoll{\ell_k}] \bigg)_{k \in \N_0}
   \quad\text{and}\quad
   \bigg( \sum\limits_{\nnu \in \rmarg_{\ell_k}} \tau_{\ell_k\nnu}[\ssolhl{\ell_k}] \bigg)_{k \in \N_0}
\] 
converge to zero.}
\end{theorem}
\begin{proof}
We will prove the convergence result for $\tau_{\ell_k\nnu}[\ssoll{\ell_k}]$,
while the proof for $\tau_{\ell_k\nnu}[\ssolhl{\ell_k}]$ is exactly the same.
We will use the analyticity of the semidiscrete approximation $\ssoll{\ell_k}: \G \to \X_{\ell_k}$
(this property follows from~\cite[Lemma~3.2]{BabuskaNT07};
see also the proof of Lemma~\ref{lemma:summability:general} and Remark~\ref{rem:summability:general}).
Specifically,
the assumptions on the diffusion coefficient in 
either Lemma~\ref{lemma:summability:affine} or Lemma~\ref{lemma:summability:general} 
guarantee that $\ssoll{\ell_k}(\cdot,\y)$
admits an analytic extension in the region
$\Sigma(\G, \boldsymbol{\sigma}) =
  \left\{ \boldsymbol{\zeta} \in \mathbb{C}^M, \; \text{dist}(\zeta_m, \G_m)\le \sigma_m, \; m=1,\dots, M\right\}
$
for some $\boldsymbol{\sigma} = [\sigma_1, \dots, \sigma_M]$;
furthermore,
$\max\limits_{\boldsymbol{\zeta} \in \Sigma(\G, \boldsymbol{\sigma})}
  \norm{\ssoll{\ell_k}(\cdot, \boldsymbol{\zeta})}{\X} \le C_{\rm reg}$
with a positive constant $C_{\rm reg}$ that depends on the problem data and is independent of the discretization in the spatial domain.
Hence, applying Lemma~\ref{lemma:interpolators} and then~\cite[Lemma 2.2]{FeischlS21},
we obtain for any $\nnu \in \rmarg_{\ell_k}$:
\begin{align} \label{eq:surplus:norm}
    \tau_{\ell_k\nnu}[\ssoll{\ell_k}]
    & \reff{eq:param:indicator1}=
    \norm[\bigg]{\Delta^{\bmf(\nnu)} \sum\limits_{\mmu \in \indset_{\ell_k} \cup \rmarg_{\ell_k}} \Delta^{\bmf(\mmu)} \ssoll{\ell_k}}{} =
    \norm[\bigg]{\Delta^{\bmf(\nnu)} \sum\limits_{\mmu \in \N^M} \Delta^{\bmf(\mmu)} \ssoll{\ell_k}}{} \reff{eq:poly:representation}=
    \norm[\big]{ \Delta^{\bmf(\nnu)} \ssoll{\ell_k}}{}
    \nonumber
    \\[4pt]
    & \lesssim
    \leb(\nnu) \, e^{-\beta \|\bmf(\nnu-\1)\|_1} \,
    \max\limits_{\boldsymbol{\zeta} \in \Sigma(\G, \boldsymbol{\sigma})} \norm{\ssoll{\ell_k}(\cdot, \boldsymbol{\zeta})}{\X} \le
    C_{\rm reg} \, \leb(\nnu) \, e^{-\beta \|\bmf(\nnu-\1)\|_1},
\end{align}
where $\beta:=\min\limits_{m = 1, \dots, M} \beta_m > 0$ with
$\beta_m := \log{\left(\frac{2\sigma_m}{|\G_m|} + \sqrt{1+\frac{4\sigma_m^2}{|\Gamma_m|^2}} \right)}$,
$\leb(\nnu)$ is the Lebesgue constant,
and the hidden constant is independent of $\ssoll{\ell_k}$ and the discretization in the spatial domain.

Note that for any $\nnu \in \rmarg_{\ell_k}$, the Lebesgue constant $\leb(\nnu)$ can be bounded as follows:
\begin{equation*} \label{ineq:lebesque}
    \leb(\nnu) \reff{eq:lebesgue}{\lesssim}
    \begin{cases}
        \prod_{m=1}^M \nu_m \le \left( \frac{\|\nnu\|_1}{M}\right)^M \le \left( \frac{M+k+1}{M} \right)^M \lesssim (k+1)^M
        &\text{for CC points,} 
        \\[4pt]
        \prod_{m=1}^M \nu^2_m \log{\nu_m} \le \prod_{m=1}^M \nu^3_m \lesssim  (k+1)^{3M}
        & \text{for Leja points}.
    \end{cases}
\end{equation*}
Furthermore, the definition of the reduced margin implies that
\begin{equation*} \label{ineq:rmarg}
    \# \rmarg_{\ell_k} \le (k+1)^M.
\end{equation*}
Therefore, using~\eqref{eq:surplus:norm}, we obtain
\begin{align} \label{ineq:alt:param}
    \sum\limits_{\nnu \in \rmarg_{\ell_k}}
    \tau_{\ell_k\nnu}[\ssoll{\ell_k}] & 
    \lesssim
    \sum\limits_{\nnu \in \rmarg_{\ell_k}}
    \leb(\nnu) \, e^{-\beta \|\bmf(\nnu-\1)\|_1}
    \nonumber
    \\[4pt]
    & \lesssim
    \begin{cases}
        (k+1)^{2M} e^{-\beta\min\limits_{\nnu \in \rmarg_{\ell_k}} \|\bmf(\nnu-\1)\|_1 } & \text{for CC points,} 
        \\[4pt]
        (k+1)^{4M} e^{-\beta\min\limits_{\nnu \in \rmarg_{\ell_k}} \|\nnu-\1\|_1 } & \text{for Leja points}.
    \end{cases}
\end{align}
The parametric enrichment by adding any multi-index
$\nnu^*_{\ell_k} \in \rmarg_{\ell_k} \setminus \markindset_{\ell_k}$ satisfying~\eqref{eq:index:min}
ensures that $\min\limits_{\nnu \in \rmarg_{\ell_k}} \|\nnu-\1\|_1$ increases with parametric enrichments.
Let us prove this fact by estimating $\min\limits_{\nnu \in \rmarg_{\ell_k}} \|\nnu-\1\|_1$ in terms of the parametric enrichment counter $k$.
We consider the case of Leja points, and the arguments below apply immediately to CC 
points, since $\bmf$ is a bijection.
Let $b_k := \min\limits_{\nnu \in \rmarg_{\ell_k}}\|\nnu-\1\|_1$, $k \in \N_0$.
Note that for a given $b \in \N$ there are $\binom{b + M - 1}{b}$ different multi-indices $\nnu \ge \1$ such that
$\|\nnu-\1\|_1 = b$.
Therefore, by marking a multi-index $\nnu^*_{\ell_k}$ satisfying~\eqref{eq:index:min} and adding it to the index set $\indset_{\ell_k}$,
it is guaranteed that any current value of $b_k$ will increase after at most $\binom{b_k + M - 1}{b_k}$ parametric enrichments.
Consequently, the number of parametric enrichments required to reach a given value of $b_k$ can be estimated as follows:
\begin{align*}
    k+1 & \le \sum_{n=1}^{b_k} \binom{n + M - 1}{n} = \sum_{n=1}^{b_k} \frac{(n+M-1)!}{n!(M-1)!} =
    \sum_{n=1}^{b_k} \prod_{m=1}^{M-1} \frac{n+m}{m} =
    \nonumber
    \\[4pt]
    & =
    \sum_{n=1}^{b_k} \prod_{m=1}^{M-1} \left(1+\frac{n}{m}\right) \le  \sum_{n=1}^{b_k} (1+n)^{M-1} <
    \sum_{n=1}^{b_k+1} n^{M-1} \lesssim (b_k + 1)^M.
\end{align*}
Thus, $\min\limits_{\nnu \in \rmarg_{\ell_k}}\|\nnu - \1\|_1 = b_k \gtrsim \sqrt[M]{k+1}$.
Substituting this estimate into~\eqref{ineq:alt:param}, we deduce that
\begin{equation*}
     \sum\limits_{\nnu \in \rmarg_{\ell_k}} \tau_{\ell_k \nnu}[\ssoll{\ell_k}] \lesssim
     (k+1)^{4M} e^{-\beta \sqrt[M]{k+1}} \xrightarrow{k\xrightarrow{}\infty} 0.
\end{equation*}
This concludes the proof.
\end{proof}

\section{Convergence of spatial error estimates} \label{sec:spat:convergence}

In this section, we prove convergence of spatial error indicators $\mu_{\ell_k \z}$
along a subsequence $\left(\ell_k\right)_{k \in \N_0}$ of iterations where spatial refinements occur.
In fact, for a fixed collocation point~$\z$, convergence of spatial error indicators $\mu_{\ell_k \z}$ in~\eqref{eq:2level:indicator}
can be inferred from the results of~\cite{msv08} for deterministic problems.
Indeed, for each sample $\z \in \G$, the problem formulation, its discretization and the adaptive refinement process
satisfy the general framework in~\cite[section~2]{msv08}.
Specifically:
(i)~the weak formulation~\eqref{eq:pde:weak} fits into the class of problems considered in~\cite[section~2.1]{msv08};
(ii)~the Galerkin discretization~\eqref{eq:sample:fem} satisfies the assumptions in~\cite[eqs.~(2.6)--(2.8)]{msv08};
(iii)~the spatial NVB refinement
satisfies the assumptions on mesh refinement in~\cite[eqs.~(2.5)~and~(2.14)]{msv08};
(iv)~the D{\"o}rfler marking criterion~\eqref{eq:doerfler:separate1} satisfies the marking condition in~\cite[eq.~(2.13)]{msv08};
and finally,
(v)~the local error indicators~\eqref{eq:2level:local:indicator} satisfy~\cite[eq.~(2.9b)]{msv08}.
\rev{Thus, repeating the arguments in the proof of Theorem~2.1 in~\cite{msv08},
we establish the following result.}

\begin{theorem} \label{theorem:spatial}
Let $\z \in \bigcup_{\ell \in \N_0} \Colpts_\ell$ be a collocation point generated by Algorithm~{\rm \ref{algorithm}}.
Let $(\ell_k)_{k \in \N_0} \subset \N_0$ denote a subsequence of iterations
where spatial refinements occur in Algorithm~{\rm \ref{algorithm}} such that $\z \in \Colpts_{\ell_0}$ and
$\ell_k\, {\xrightarrow{k \to \infty}}\, \infty$.
Then the associated spatial error indicators $\mu_{\ell_k \z}$ converge to zero, i.e.,
$\mu_{\ell_k \z} \xrightarrow{k \to \infty} 0$.
\end{theorem}

\rev{It is important to note that the global reliability of the estimator (see~\eqref{eq:error:estimate} and~\cite[eq.~(2.9a)]{msv08})
is not needed in the proof of the above result.
The reliability is only used in~\cite[Theorem~2.1]{msv08} to prove convergence of the true finite element error.
Likewise, we will use the reliability of our error estimate to establish the convergence
of adaptively generated SC-FEM approximations to the true solution of problem~\eqref{eq:pde:strong}; see Corollary~\ref{cor:main}.}

\section{Convergence of the adaptive algorithm} \label{sec:main:results}

Now we are ready to prove the main result of this work.

\begin{theorem} \label{theorem:main}
Let $f\in L^2(D)$ and let the diffusion coefficient $a(x, \y)$ satisfy
the hypotheses of either Lemma~{\rm \ref{lemma:summability:affine}} or Lemma~{\rm \ref{lemma:summability:general}}.
Then for any choice of marking parameters $\theta_\X$, $\theta_\Colpts$ and $\vartheta$,
Algorithm~{\rm \ref{algorithm}} generates a convergent sequence of error estimates, specifically,
$\mu_\ell + \tau_\ell \to 0$ as $\ell \to \infty$.
\end{theorem}

\begin{proof}
The assumption on $a(x, \y)$
implies that the Taylor coefficients $[\ssolc]_{\ii}$, $\ii \in \N_0^M$, defined by~\eqref{eq:Taylor:coeff}
with $w = \ssolc$ satisfy the summability property~\eqref{eq:summability:property} 
(see Lemmas~\ref{lemma:summability:affine} and~\ref{lemma:summability:general}).
This, in particular, enables the application of Theorem~\ref{theorem:parametric} in the proof below,
where we consider three possible refinement scenarios that may occur when running Algorithm~\ref{algorithm}.

{\bf Scenario~1.}
    In the first scenario, the \emph{spatial} refinement occurs finitely many times, i.e., $\exists\,\ell_0 \in \N_0$ such that
    $\sum_{\z \in \Colpts_\ell} \mu_{\ell \z} \norm{L_{\ell \z}}{L^p_{\pi}(\G)} 
      < \vartheta \sum_{\nnu \in \rmarg_\ell} \tau_{\ell \nnu}$
    for all $\ell \ge \ell_0$.
    In this case, applying
    Theorem~\ref{theorem:parametric}, we~obtain
    \begin{equation} \label{eq:mu_ell:sequence}
       \tau_\ell \xrightarrow{\ell \to \infty} 0\quad \text{and}\quad
       \mu_\ell \reff{eq:err:indicators}\lesssim \sum_{\z \in \Colpts_\ell} \mu_{\ell \z} \norm{L_{\ell \z}}{L^p_{\pi}(\G)} <
       \vartheta \sum\limits_{\nnu \in \rmarg_\ell} \tau_{\ell  \nnu}\xrightarrow{\ell \to \infty} 0.
    \end{equation}

{\bf Scenario~2.}
    In the second scenario, the \emph{parametric} refinement occurs finitely many times, i.e.,
    $\exists\,\ell_0 \in \N_0$ such that
    $\sum_{\nnu \in \rmarg_\ell} \tau_{\ell \nnu} \le
      \vartheta^{-1}\sum_{\z \in \Colpts_\ell} \mu_{\ell \z} \norm{L_{\ell \z}}{L^p_{\pi}(\G)}$
    for all $\ell \ge \ell_0$.
    In this case, starting with iteration $\ell = \ell_0$, the algorithm performs only spatial refinements.
    Therefore, the set of collocation points $\Colpts_\ell$ (and, consequently, the associated set of Lagrange polynomials) 
    stays the same for all iterations $\ell \ge \ell_0$.
    Thus, applying Theorem~\ref{theorem:spatial},
    we conclude that
    \[
       \tau_\ell \le \sum_{\nnu \in \rmarg_\ell} \tau_{\ell \nnu} \le
       \vartheta^{-1}\sum_{\z \in \Colpts_\ell} \mu_{\ell \z} \norm{L_{\ell \z}}{L^p_{\pi}(\G)}  \xrightarrow{\ell \to \infty} 0
       \quad \text{and}\quad \mu_\ell \xrightarrow{\ell \to \infty} 0.
    \]

{\bf Scenario~3.}
    Finally, both types of refinement may occur infinitely often.
    In this case,
    we split the sequences $\left(\mu_\ell\right)_{\ell \in \N_0}$ and $\left(\tau_\ell\right)_{\ell \in \N_0}$ into disjoint subsequences
    as follows:
    \[
       \left(\mu_\ell\right)_{\ell \in \N_0} = \big(\mu_{\ell^{(a)}_k}\big)_{k \in \N_0} \cup \big(\mu_{\ell^{(b)}_k}\big)_{k \in \N_0} \quad
       \text{and}\quad
       \left(\tau_\ell\right)_{\ell \in \N_0} = \big(\tau_{\ell^{(a)}_k}\big)_{k \in \N_0} \cup \big(\tau_{\ell^{(b)}_k}\big)_{k \in \N_0};
    \]
    here, the subsequences indexed by $(a)$ (resp., by $(b)$) correspond to iterations
    where only the spatial (resp., parametric) refinement occurs (see Figure~\ref{fig:subseq},
    where we denote by
    $\bar\mu_{\ell} := \sum_{\z \in \Colpts_\ell} \mu_{\ell \z} \norm{L_{\ell \z}}{L^p_{\pi}(\G)}$ 
    (resp., $\bar\tau_{\ell} := \sum_{\nnu \in \rmarg_\ell} \tau_{\ell \nnu}$)
    \rev{the weighted sum of spatial (resp., parametric) error indicators}
    at the $\ell$-th iteration).

\begin{figure}[!thp]
\centering
\begin{tikzpicture}
\pgfplotstableread{subsequences_example.dat}{\one}
\begin{loglogaxis}
[
width = 7.2cm, height = 7.2cm,						
ytick={0.1,0.01},
xlabel={degrees of freedom (dof)}, 					
xlabel style={font=\fontsize{10pt}{12pt}\selectfont},		
ylabel={\rev{weighted sum of error indicators}},				
ylabel style={font=\fontsize{10pt}{12pt}\selectfont}, 		
ymajorgrids=true, xmajorgrids=true, grid style=dashed,	
xmin=2*10^4, xmax=1.1*10^6,						
ymin = 0.03,	 ymax = 0.3,							
legend style={legend pos=south west, legend cell align=left, fill=none}
]
%
\addplot[blue,mark=square,mark size=2pt]		table[x=dofs, y=error_s]{\one};
\addplot[red,mark=diamond,mark size=2.6pt]		table[x=dofs, y=error_p]{\one};
\draw (axis cs:24246, 0.244354033) +(0pt,5pt) node[right]{$\ell_0^{(a)} = 0$};
\draw [<-] (axis cs:33210, 0.083353904) -- +(0pt,-18pt) node{$\ell_1^{(a)} = 1$};
\draw (axis cs:44838, 0.177516045) +(0pt,5pt) node[right]{$\ell_2^{(a)} = 2$};
\draw [<-](axis cs:61929, 0.083353904) -- +(0pt,-30pt) node{$\ell_3^{(a)} = 3$};
\draw (axis cs:90684, 0.131570124) +(0pt,5pt) node[right]{$\ell_4^{(a)} = 4$};
\draw [<-] (axis cs:116883,0.083353904) -- +(0pt,-18pt) node{$\ell_5^{(a)} = 5$};
\draw (axis cs:155007,0.096908271) +(0pt,5pt) node[right]{$\ell_6^{(a)} = 6$};
\draw [<-] (axis cs:210015,0.083353904) -- +(0pt,-30pt) node{$\ell_7^{(a)} = 7$};
\draw (axis cs:301473,0.072608466)  +(2pt,3pt) node[right]{$\ell_0^{(b)} = 8$};
\draw (axis cs:502455,0.134431179) +(2pt,10pt) node{$\ell_8^{(a)} = 9$};
\draw [<-] (axis cs:715155,0.042838902) -- +(-4pt,-10pt) node{$\;\ell_{9}^{(a)} = 10$};
\legend{
$\bar\mu_\ell$ (spatial),
$\bar\tau_\ell$ (parametric)
}
\end{loglogaxis}
\end{tikzpicture}
\caption{\rev{An example of iteration subsequences
illustrating Scenario~3 in the proof of Theorem~\ref{theorem:main}.}}
\label{fig:subseq}
\end{figure}

        For the subsequences $\big(\tau_{\ell^{(b)}_k}\big)_{k \in \N_0}$ and
        $\big(\mu_{\ell^{(b)}_k}\big)_{k \in \N_0}$,
        arguing as in~\eqref{eq:mu_ell:sequence} we conclude~that
        \begin{equation}\label{eq:subseq:b}        
           \tau_{\ell^{(b)}_k}  \xrightarrow{k \to \infty} 0\quad \text{and}\quad
           \mu_{\ell^{(b)}_k} \lesssim \sum_{\z \in \Colpts_{\ell^{(b)}_k}} \mu_{\ell^{(b)}_k \z}
           \norm[\big]{L_{{\ell^{(b)}_k} \z}}{L^p_{\pi}(\G)} <
           \vartheta \sum\limits_{\nnu \in \rmarg_{\ell^{(b)}_k}} \tau_{{\ell^{(b)}_k} \nnu} \xrightarrow{k \to \infty} 0.
        \end{equation}
Thus, it remains to show that $\mu_{\ell^{(a)}_k} \xrightarrow{k \to \infty} 0$  and  $\tau_{\ell^{(a)}_k} \xrightarrow{k \to \infty} 0$.
For any $k \in \N_0$, we denote by $q = q(k) \in \N_0$ and $n = n(k) \in \N$ the smallest possible integers
such that $\ell^{(a)}_k + n = \ell^{(b)}_{q}$
(such values of $q$ and $n$ always exist, since both types of refinement occur infinitely many times;
furthermore, $q \to \infty$ as $k \to \infty$).
We split the rest of the proof into two steps.

{\bf Step~3.1.}
        Firstly, if for several consecutive $k \in \N$, we have $\ell^{(a)}_{k+1} = \ell^{(a)}_{k} + 1$,
        it means that a number of spatial refinements occur sequentially;
        e.g., in Figure~\ref{fig:subseq}, this corresponds to iterations $\ell^{(a)}_k$ with $k=0, \dots, 7$.
        Thus, the set of collocation points $\Colpts_{\ell^{(a)}_{k+i}}$ and the associated set of Lagrange polynomials remain the same for
        $i = 0, 1, \dots, n$.
        Due to Theorem~\ref{theorem:spatial} and the nature of the scenario we are considering,
        the \rev{weighted} sum of spatial indicators,
        $\bar\mu_{\ell_k^{(a)}} = \sum_{\z \in \Colpts_{\ell^{(a)}_k}} \mu_{{\ell^{(a)}_k} \z} \, \norm{L_{{\ell^{(a)}_k} \z}}{L^p_{\pi}(\G)}$,
        will eventually fall below the following threshold as $k$ increases: 
        \begin{equation} \label{eq:prescribed:value}  
        \vartheta \!\! \sum\limits_{\nnu \in \rmarg_{\ell^{(a)}_{k}}} \tau_{\ell^{(a)}_{k} \nnu} \,{=}\,
        \vartheta \!\! \sum\limits_{\nnu \in \rmarg_{\ell^{(a)}_{k}+1}} \tau_{(\ell^{(a)}_{k}+1)  \nnu} \,{=}\, \ldots \,{=}\,
        \vartheta \!\! \sum\limits_{\nnu \in \rmarg_{\ell^{(a)}_{k}+n}} \tau_{(\ell^{(a)}_{k}+n)  \nnu} \,{=}\,
        \vartheta \!\! \sum\limits_{\nnu \in \rmarg_{\ell^{(b)}_{q}}} \tau_{\ell^{(b)}_{q}  \nnu} \xrightarrow{q \to \infty} 0
        \end{equation}
        (here, the equality is ensured by the fact that $\tau_{\ell \nnu}$ are independent of mesh refinements,
        since they are calculated using the coarsest-mesh Galerkin approximations; cf.~\eqref{eq:param:indicator}).
        This triggers the change of the refinement type from spatial to parametric, i.e.,        
        \begin{align*}
            \sum_{\z \in \Colpts_{\ell^{(a)}_k+n}} \mu_{({\ell^{(a)}_k + n}) \z} \, \norm[\big]{L_{({\ell^{(a)}_k+n}) \z}}{L^p_{\pi}(\G)} = 
            \sum_{\z \in \Colpts_{\ell^{(b)}_{q}}} \mu_{{\ell^{(b)}_{q}} \z} \, \norm[\big]{L_{{\ell^{(b)}_{q}} \z}}{L^p_{\pi}(\G)} <
            \vartheta \sum\limits_{\nnu \in \rmarg_{\ell^{(b)}_{q}}} \tau_{\ell^{(b)}_{q}\nnu}  \xrightarrow{q \to \infty} 0.
        \end{align*}
        
{\bf Step~3.2.}
       Next, let us consider the case when $\ell^{(a)}_{k+1} > \ell^{(a)}_{k} + 1$,
       i.e., at least one parametric enrichment occurs between two spatial refinements
       (for example, see iterations $\ell_7^{(a)} = 7$, $\ell_8^{(a)} = 9$ and $\ell_0^{(b)} = 8$ in Figure~\ref{fig:subseq}).
       We will show that in this case the spatial error estimate $\mu_{\ell^{(a)}_{k+1}}$ is bounded
       by a quantity that converges to zero as $k \to \infty$.
       Using the definition of spatial error estimates in~\eqref{eq:spatial:estimate} and the marking criterion~\eqref{eq:param:marking}
       for parametric enrichment, we obtain
       \begin{align} \label{eq:error:est:upper:bound}
           \mu_{\ell^{(a)}_{k+1}}  & = 
           \norm[\Big]{S_{\ell^{(a)}_{k+1}} \Big( \widehat U_{\ell^{(a)}_{k+1}} - U_{\ell^{(a)}_{k+1}} \Big)}{} \reff{eq:S:def} =
           \norm[\bigg]{ \sum\limits_{\nnu \in \indset_{\ell^{(b)}_{q}} \cup \markindset_{\ell^{(b)}_{q}}\cup \{\nnu^*_{\ell^{(b)}_{q}}\}}
                       \Delta^{\bmf(\nnu)} \big( \ssolhl{\ell^{(a)}_{k+1}} - \ssoll{\ell^{(a)}_{k+1}} \big) }{}
           \nonumber
            \\[4pt]
            & 
            \le \norm[\bigg]{ \sum\limits_{\nnu \in \indset_{\ell^{(b)}_{q}}}
                             \Delta^{\bmf(\nnu)} \big( \ssolhl{\ell^{(a)}_{k+1}} - \ssoll{\ell^{(a)}_{k+1}} \big) }{} +
            \sum\limits_{\nnu \in \rmarg_{\ell^{(b)}_{q}}}
            \norm[\Big]{\Delta^{\bmf(\nnu)} \big( \ssolhl{\ell^{(a)}_{k+1}} - \ssoll{\ell^{(a)}_{k+1}} \big)}{}, 
        \end{align}
        where $\ssoll{\ell^{(a)}_{k+1}}$ and $\ssolhl{\ell^{(a)}_{k+1}}$ are semidiscrete solutions satisfying~\eqref{eq:semidiscrete}
        with $\W = \X_{\ell^{(a)}_{k+1}}$ and $\W = \widehat\X_{\ell^{(a)}_{k+1}}$, respectively.
        Note that $\ssoll{\ell^{(a)}_{k+1}} = \ssoll{\ell^{(b)}_{q}}$, as the finite element mesh does not change
        during the parametric enrichment step.
        Therefore, the first term on the right-hand side of~\eqref{eq:error:est:upper:bound} is an element of the sequence
        $(\mu_{\ell^{(b)}_k} )_{k\in \N_0}$, for which we have already proved convergence; cf.~\eqref{eq:subseq:b}.
        Thus,
        \begin{equation} \label{eq:spat:b:converge}
            \norm[\bigg]{\sum\limits_{\nnu \in \indset_{\ell^{(b)}_{q}}}
                                                                                    \Delta^{\bmf(\nnu)} \big( \ssolhl{\ell^{(b)}_{q}} - \ssoll{\ell^{(b)}_{q}} \big)}{}
            = \mu_{\ell^{(b)}_{q}} \xrightarrow{q \to \infty} 0.
        \end{equation}
        The second term on the right-hand side of~\eqref{eq:error:est:upper:bound} can be estimated using the triangle inequality:
        \begin{equation}\label{eq:2:alt:param}
             \sum\limits_{\nnu \in \rmarg_{\ell^{(b)}_{q}}} \norm[\Big]{\Delta^{\bmf(\nnu)} \big( \ssolhl{\ell^{(b)}_{q}} - \ssoll{\ell^{(b)}_{q}} \big)}{} \le
             \sum\limits_{\nnu \in \rmarg_{\ell^{(b)}_{q}}}
             \norm[\Big]{\Delta^{\bmf(\nnu)} \ssolhl{\ell^{(b)}_{q}}}{} +
             \sum\limits_{\nnu \in \rmarg_{\ell^{(b)}_{q}}} \norm[\Big]{\Delta^{\bmf(\nnu)}\ssoll{\ell^{(b)}_{q}}}{}.
        \end{equation}
        \rev{For the first sum on the right-hand side of~\eqref{eq:2:alt:param},
        using Lemma~\ref{lemma:interpolators} and the definition in~\eqref{eq:param:indicator1}
        with $w = \ssolhl{\ell^{(b)}_{q}}$, we find that
        \[
            \sum\limits_{\nnu \in \rmarg_{\ell^{(b)}_{q}}} \norm[\Big]{\Delta^{\bmf(\nnu)}\ssolhl{\ell^{(b)}_{q}}}{}
            \reff{eq:poly:representation} =
            \sum\limits_{\nnu \in \rmarg_{\ell^{(b)}_{q}}}
            \norm[\bigg]{\Delta^{\bmf(\nnu)}
            \sum\limits_{\mmu \in \indset_{\ell^{(b)}_{q}} \cup \rmarg_{\ell^{(b)}_{q}}}\Delta^{\bmf(\mmu)} \ssolhl{\ell^{(b)}_{q}}}{}
            \reff{eq:param:indicator1} =
            \sum\limits_{\nnu \in \rmarg_{\ell^{(b)}_{q}}} \tau_{\ell^{(b)}_{q}\nnu} \big[ \ssolhl{\ell^{(b)}_{q}} \big].
       \]
       Thus, applying Theorem~\ref{theorem:parametic:alt}, we conclude that
        \begin{equation} \label{eq:alt:converge}
            \sum\limits_{\nnu \in \rmarg_{\ell^{(b)}_{q}}} \norm[\Big]{\Delta^{\bmf(\nnu)}\ssolhl{\ell^{(b)}_{q}}}{}
            \xrightarrow{q \to \infty} 0.
       \end{equation}
    The same arguments apply to the second sum on the right-hand side of~\eqref{eq:2:alt:param}.}
    
    From~\eqref{eq:error:est:upper:bound}--\eqref{eq:alt:converge}
    we conclude that $\mu_{\ell^{(a)}_{k+1}} \xrightarrow{k \to \infty} 0$.
    Furthermore, it follows from~\eqref{eq:prescribed:value} and~\eqref{eq:err:indicators} that
    $\tau_{\ell^{(a)}_{k+1}} \le \sum_{\nnu \in \rmarg_{\ell_{k+1}^{(a)}}} \tau_{{\ell^{(a)}_{k+1}} \nnu} \xrightarrow{k \to \infty} 0$.
    Thus, we have proved that all considered subsequences converge to zero as $k \to \infty$. 
    Hence,
    $\mu_{\ell} + \tau_{\ell} \xrightarrow{\ell \to \infty} 0.$

For each refinement scenario, we have established convergence of spatial and parametric error estimates to zero.
This concludes the proof of the theorem.
\end{proof}

The following result is an immediate consequence of Theorem~\ref{theorem:main} and
the a posteriori error estimate in~\eqref{eq:error:estimate}.

\begin{corollary} \label{cor:main}
Let $f\in L^2(D)$ and let the diffusion coefficient $a(x, \y)$ satisfy
the hypotheses of either Lemma~{\rm \ref{lemma:summability:affine}} or Lemma~{\rm \ref{lemma:summability:general}}.
Let $\big(u_{\ell}^{\rm SC}\big)_{\ell \in \N_0}$ be the sequence of SC-FEM approximations generated by Algorithm~{\rm \ref{algorithm}} and
denote by $\big(\widehat u_{\ell}^{\rm SC}\big)_{\ell \in \N_0}$ the associated sequence of enhanced SC-FEM approximations
(as described in section~{\rm \ref{sec:error:estimate}}).
Suppose that the saturation assumption~{\rm \eqref{eq:saturation}} holds for each pair
$u_{\ell}^{\rm SC},\, \widehat u_{\ell}^{\rm SC}$ ($\ell \in \N_0$).
Then for any choice of marking parameters $\theta_\X$, $\theta_\Colpts$ and $\vartheta$,
the sequence of SC-FEM approximations converges to the true solution of problem~\eqref{eq:pde:strong},
i.e., $\norm{u - u_{\ell}^{\rm SC}}{} \to 0$ as $\ell \to \infty$.
\end{corollary}

\section{Numerical results} \label{sec:numerics}

In this section, we present the numerical results that underpin our theoretical findings.
These results were generated using \rev{the open-source MATLAB toolbox Adaptive ML-SCFEM~\cite{BespalovSX_a_ml_scfem}
on an Intel Core i5-6500 3.20GHz CPU with 16GB of RAM}.

For each of the two test cases described below, we set an error tolerance and run Algorithm~\ref{algorithm} with the stopping criterion
$\mu_\ell + \tau_\ell < \tt{error tolerance}$, where $\mu_\ell$ and $\tau_\ell$ are the spatial and parametric error estimates, respectively
(see~\eqref{eq:spatial:estimate},~\eqref{eq:param:estimate}).
In all experiments, we employ the marking strategy in Algorithm~\ref{algorithm_m} with marking parameters
$\vartheta =1$, $\theta_\X = \theta_\Colpts = 0.3$.
In particular, the type of refinement is determined in Algorithm~\ref{algorithm_m} by comparing
\rev{the weighted sums of spatial and parametric error indicators, i.e.,}
$\bar\mu_{\ell} = \sum_{\z \in \Colpts_\ell} \mu_{\ell \z} \norm{L_{\ell \z}}{L^p_{\pi}(\G)}$
and $\bar\tau_{\ell}= \sum_{\nnu \in \rmarg_\ell} \tau_{\ell \nnu}$.

In both test cases, the parameters $y_m, \; m=1, \dots, M$ are the images of uniformly 
distributed independent mean-zero random variables, so that 
$\dpi_m(y_m) = \frac{1}{2} \mathrm{d}y_m$.
We will present the results for both Leja and Clenshaw-Curtis sets of collocation points in each test case.

\subsection{Test case I:  affine coefficient data (cookie problem)}\label{sec:affineresults}

Our first example is the test problem considered in~\cite[section~4.2]{FeischlS21}.
Let $D = (0,1)^2$ and let $F,\, A_1,\, A_2,\, \ldots,\, A_8$ be nine disjoint subdomains of $D$ as depicted
in the left plot of Figure~\ref{fig:spatial:domain}.
We set the forcing term as the characteristic function of $F$, i.e., $f(x) = 100\chi_F(x)$,
and look to solve the model problem~\eqref{eq:pde:strong} with the parametric coefficient given by
\begin{align} \label{kl}
   a(x, \y) = a_0(x) + \sum_{m = 1}^8 a_m(x) \, y_m,\quad
   x \in D,\ \y \in \Gamma.
\end{align}
Following~\cite[section~4.2]{FeischlS21}, we set the expansion coefficients as 
\begin{equation}
    a_0(x) \equiv 1.1 \; \text{and} \; a_m(x) = \omega_m \chi_m(x)\; \text{for} \; m=1, \dots, 8,
\end{equation}
where $\{\omega_m\}_{m=1}^8 = \{1, 0.8, 0.4, 0.2, 0.1, 0.05, 0.02, 0.01\}$ and
$\chi_m(x)$ is the characteristic function of subdomain $A_m$.

\begin{figure}[!th]
\centering
\begin{tikzpicture}[thick,scale=6, every node/.style={scale=1}]
\draw (0.1, 0.1) -- (0.1, 0.3) -- (0.3, 0.3) -- (0.3, 0.1) -- cycle;
\draw (0.4, 0.1) -- (0.4, 0.3) -- (0.6, 0.3) -- (0.6, 0.1) -- cycle;
\draw (0.7, 0.1) -- (0.7, 0.3) -- (0.9, 0.3) -- (0.9, 0.1) -- cycle;
\draw (0.1, 0.4) -- (0.1, 0.6) -- (0.3, 0.6) -- (0.3, 0.4) -- cycle;
\draw (0.4, 0.4) -- (0.4, 0.6) -- (0.6, 0.6) -- (0.6, 0.4) -- cycle;
\draw (0.7, 0.4) -- (0.7, 0.6) -- (0.9, 0.6) -- (0.9, 0.4) -- cycle;
\draw (0.1, 0.7) -- (0.1, 0.9) -- (0.3, 0.9) -- (0.3, 0.7) -- cycle;
\draw (0.4, 0.7) -- (0.4, 0.9) -- (0.6, 0.9) -- (0.6, 0.7) -- cycle;
\draw (0.7, 0.7) -- (0.7, 0.9) -- (0.9, 0.9) -- (0.9, 0.7) -- cycle;
\draw (0, 1) -- (1, 1) -- (1, 0);
\node at (0.2,0.2) {$A_1$};
\node at (0.5,0.2) {$A_2$};
\node at (0.8,0.2) {$A_3$};
\node at (0.2,0.5) {$A_4$};
\node at (0.5,0.5) {$F$};
\node at (0.8,0.5) {$A_5$};
\node at (0.2,0.8) {$A_6$};
\node at (0.5,0.8) {$A_7$};
\node at (0.8,0.8) {$A_8$};
\draw[thick,->] (0,0) -- (1.1,0) node[anchor=north west] {$x_1$};
\draw[thick,->] (0,0) -- (0,1.1) node[anchor=south east] {$x_2$};
\foreach \x in {0,0.1,0.3,0.5,0.7,0.9,1}
   \draw (\x cm,1pt) -- (\x cm,-1pt) node[anchor=north] {$\x$};
\foreach \y in {0,0.1,0.3,0.5,0.7,0.9,1}
    \draw (1pt,\y cm) -- (-1pt,\y cm) node[anchor=east] {$\y$};
\end{tikzpicture}
\includegraphics[width = 0.45\textwidth, trim=0cm 4cm 0cm 4cm,clip]{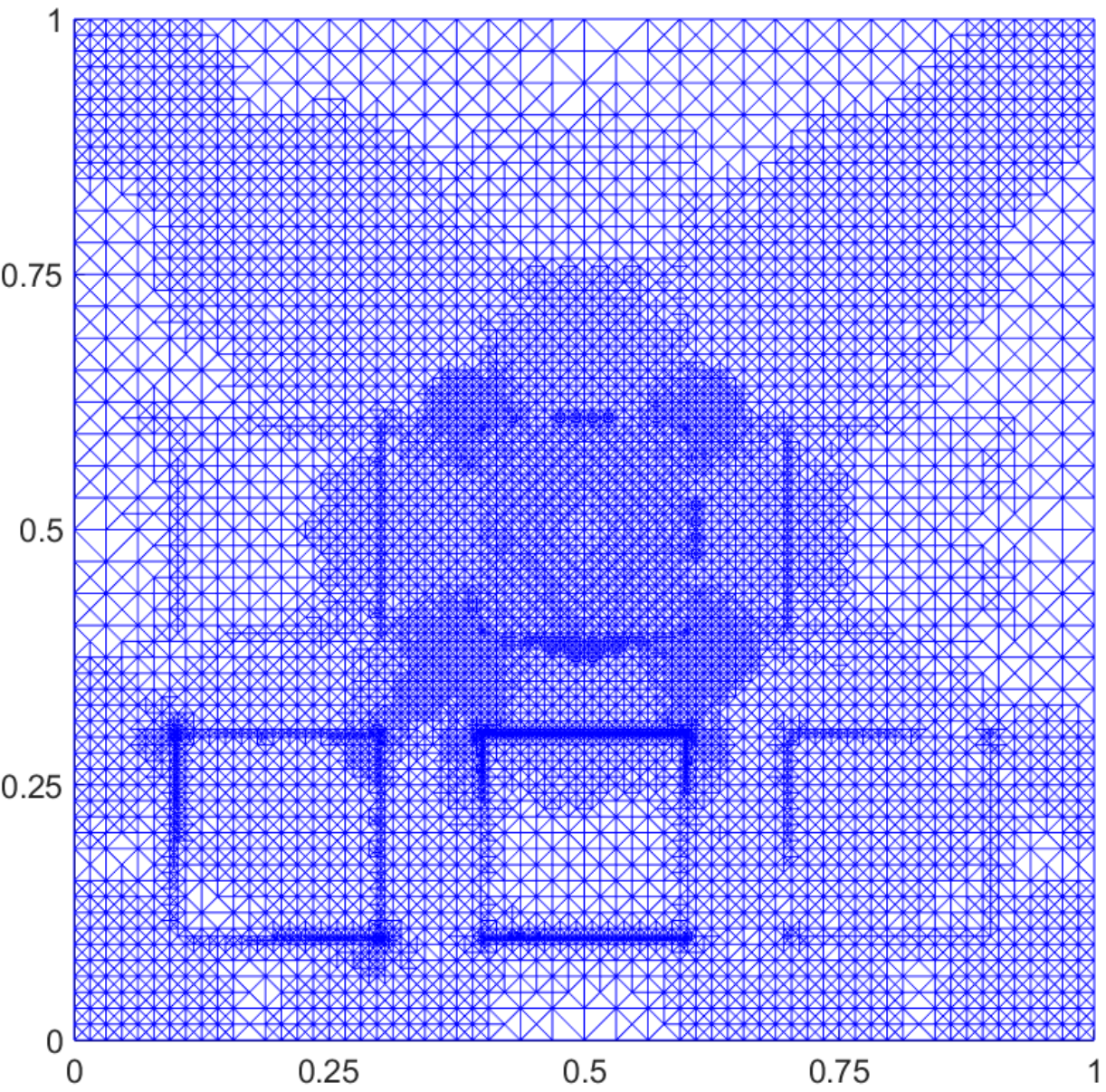}
\caption{Test case I: spatial domain and subdomains (left) and the refined mesh after 18 iterations of Algorithm~\ref{algorithm}
using Leja collocation points (right).}
\label{fig:spatial:domain}
\end{figure}

\begin{figure}[!th]
\begin{tikzpicture}
\pgfplotstableread{sl_sc_test1Leja.dat}{\one}
\begin{loglogaxis}
[
width = 7.5cm, height = 7.5cm,						
ytick={1, 0.1, 0.01},
xlabel={degrees of freedom (dof)}, 					
xlabel style={font=\fontsize{10pt}{12pt}\selectfont},		
ylabel={\rev{weighted sum of error indicators}},				
ylabel style={font=\fontsize{10pt}{12pt}\selectfont}, 		
ymajorgrids=true, xmajorgrids=true, grid style=dashed,	
xmin=81, xmax=10399856,						
ymin = 0.012,	 ymax = 0.91,							
legend style={legend pos=south west, legend cell align=left, fill=none}
]
\addplot[black,mark=o,mark size=2.2pt]		table[x=dofs, y=error]{\one};
\addplot[blue,mark=square,mark size=2pt]		table[x=dofs, y=error_s]{\one};
\addplot[red,mark=diamond,mark size=2.6pt]		table[x=dofs, y=error_p]{\one};
\legend{
$\bar\eta_\ell$ (total),
$\bar\mu_\ell$ (spatial),
$\bar\tau_\ell$ (parametric)
}
\end{loglogaxis}
\end{tikzpicture}
\quad
\begin{tikzpicture}
\pgfplotstableread{sl_sc_test1CC.dat}{\two}
\begin{loglogaxis}
[
width = 7.5cm, height = 7.5cm,						
ytick={1, 0.1, 0.01},
xlabel={degrees of freedom (dof)}, 					
xlabel style={font=\fontsize{10pt}{12pt}\selectfont},		
ylabel={\rev{weighted sum of error indicators}},				
ylabel style={font=\fontsize{10pt}{12pt}\selectfont}, 		
ymajorgrids=true, xmajorgrids=true, grid style=dashed,	
xmin=81, xmax=10399856,						
ymin = 0.012,	 ymax = 0.91,							
legend style={legend pos=south west, legend cell align=left, fill=none}
]
\addplot[black,mark=o,mark size=2.2pt]		table[x=dofs, y=error]{\two};
\addplot[blue,mark=square,mark size=2pt]		table[x=dofs, y=error_s]{\two};
\addplot[red,mark=diamond,mark size=2.6pt]		table[x=dofs, y=error_p]{\two};
\legend{
$\bar\eta_\ell$ (total),
$\bar\mu_\ell$ (spatial),
$\bar\tau_\ell$ (parametric)
}
\end{loglogaxis}
\end{tikzpicture}
\caption{Test case I: evolution of the \rev{weighted sums of error indicators} for Leja (left) and CC~(right) points.
The axes limits  are identical in the left and right plots.}
\label{fig:test1:indicators}
\end{figure}

\begin{figure}[!th]
\begin{tikzpicture}
\pgfplotstableread{sl_sc_test1Leja.dat}{\one}
\begin{loglogaxis}
[
width = 7.5cm, height = 7.5cm,						
xlabel={degrees of freedom (dof)}, 					
xlabel style={font=\fontsize{10pt}{12pt}\selectfont},		
ylabel={estimated error},				
ylabel style={font=\fontsize{10pt}{12pt}\selectfont}, 		
ymajorgrids=true, xmajorgrids=true, grid style=dashed,	
xmin=81, xmax=10399856,						
ymin = 0.005,	 ymax = 0.92,							
legend style={legend pos=south west, legend cell align=left, fill=none}
]
\addplot[black,mark=o,mark size=2.2pt]		table[x=dofs, y=error_d]{\one};
\addplot[blue,mark=square,mark size=2pt]		table[x=dofs, y=error_s_d]{\one};
\addplot[red,mark=diamond,mark size=2.6pt]		table[x=dofs, y=error_p_d]{\one};
\legend{
$\eta_\ell$ (total),
$\mu_\ell$ (spatial),
$\tau_\ell$ (parametric)
}
\end{loglogaxis}
\end{tikzpicture}
\quad
\begin{tikzpicture}
\pgfplotstableread{sl_sc_test1CC.dat}{\two}
\begin{loglogaxis}
[
width = 7.5cm, height = 7.5cm,						
xlabel={degrees of freedom (dof)}, 					
xlabel style={font=\fontsize{10pt}{12pt}\selectfont},		
ylabel={estimated error},				
ylabel style={font=\fontsize{10pt}{12pt}\selectfont}, 		
ymajorgrids=true, xmajorgrids=true, grid style=dashed,	
xmin=81, xmax=10399856,						
ymin = 0.005,	 ymax = 0.92,							
legend style={legend pos=south west, legend cell align=left, fill=none}
]
\addplot[black,mark=o,mark size=2.2pt]		table[x=dofs, y=error_d]{\two};
\addplot[blue,mark=square,mark size=2pt]		table[x=dofs, y=error_s_d]{\two};
\addplot[red,mark=diamond,mark size=2.6pt]		table[x=dofs, y=error_p_d]{\two};
\legend{
$\eta_\ell$ (total),
$\mu_\ell$ (spatial),
$\tau_\ell$ (parametric)
}
\end{loglogaxis}
\end{tikzpicture}
\caption{Test case I: evolution of the error estimates for Leja (left) and CC (right)  points.
The axes limits  are identical in the left and right plots.}
\label{fig:test1:errors}
\end{figure}
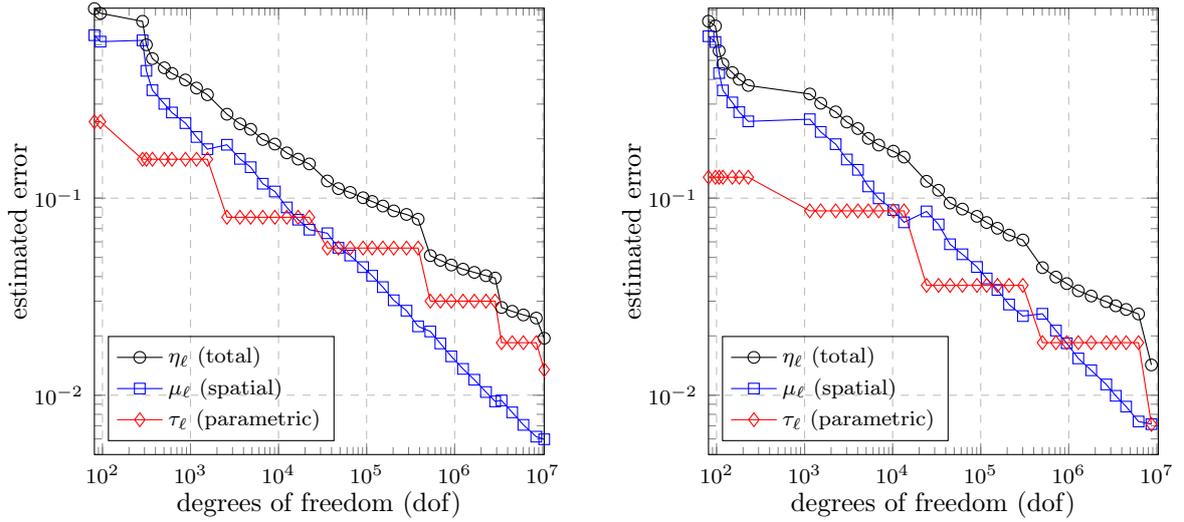

We run Algorithm~\ref{algorithm} with the initial  mesh $\TT_0$ (a uniform partition  of $D$ containing 128 right-angled triangles)
and  with $\tt{error tolerance}$ set to {\tt 2e-2}.
The right plot in Figure~\ref{fig:spatial:domain} shows the finite element mesh after 18 iterations of Algorithm~\ref{algorithm}
using Leja collocation points in the parameter domain.
This mesh is locally refined to resolve singularities at the corners of $D$ and at the boundaries of subdomains.
The magnitudes of $\{a_m(x)\}_{m=1}^8$ and $f(x)$ affect the priority and the strength of refinement around the edges of subdomains.

The tolerance was satisfied after 32 spatial refinement steps and 6 parametric enrichment steps (38 iterations in total) for Leja points and
after 30 spatial and 4 parametric enrichment steps (34 iterations in total) for CC points. 
\rev{The evolution of the weighted sums of error indicators is presented in Figure~\ref{fig:test1:indicators}, whereas
the evolution of error estimates is reported in Figure~\ref{fig:test1:errors}.}
The key point here is that the 
combined error estimate  $\eta_\ell := \mu_\ell  + \tau_\ell$ decreases 
at every iteration.
In contrast, the total \rev{weighted sum of all error indicators,} $\bar\eta_{\ell} := \bar\mu_{\ell}  + \bar\tau_{\ell}$\rev{,}
can be seen to {\it increase} at iterations that follow a parametric enrichment; see Figure~\ref{fig:test1:indicators}.
This is caused by a `jump' of the spatial error indicator $\bar\mu_{\ell}$ that occurs every time when the set of collocation point expands.
In addition to this, the mesh assigned to the new collocation point may be unsuitable for the sample of the diffusion coefficient at this point,
which causes the growth of the two-level spatial error indicators associated with new collocation points in comparison with
the spatial error indicators associated with previous collocation points.
In the experiments we carried out, we have not observed `jumps' in \emph{combined error estimates} $\eta_\ell$;
see, e.g., Figure~\ref{fig:test1:errors}.
However, as we proved in Theorem~\ref{theorem:main}, even if such `jumps' occur,
they are bounded by the terms converging to zero.

\subsection{Test case II:  nonaffine coefficient data} \label{sec:nonaffineresults}

In this case, we set $f = 1$ and solve the model problem~\eqref{eq:pde:strong}
with coefficient $a(x, \y) = \exp(h( x, \y))$ on the L-shaped domain $D = (-1, 1)^2\backslash (-1, 0]^2$.
We set the exponent field $h(x, \y)$ to have affine dependence on parameters $y_m$, i.e.,
\begin{align} \label{kll}
h(x, \y) = h_0(x) + \sum_{m = 1}^M h_m(x) \, y_m,\quad
   x \in D,\ \y \in \Gamma.
\end{align}
The expansion coefficients $h_m$, $m =0, 1, \dots, M$,  are chosen
to represent planar Fourier modes of increasing total order.
Thus, we fix  $h_0(x) := 1$ and set
\begin{equation}
\label{diff_coeff_Fourier}
h_m(x) := \alpha_m \cos(2\pi\beta_1(m)\,x_1) \cos(2\pi\beta_2(m)\,x_2),\  x=(x_1,x_2) 
\in D.
\end{equation}
The modes are ordered so that for any $m \in \N$,
\begin{equation}
  \beta_1(m) = m - k(m)(k(m)+1)/2\ \ \hbox{and}\ \ \beta_2(m) =k(m) - \beta_1(m)
\end{equation}
with $k(m) = \lfloor -1/2 + \sqrt{1/4+2m}\rfloor$.
Furthermore, to ensure that the diffusion coefficient satisfies~\eqref{eq:nonaffine},
the amplitudes $\alpha_m$ in~\eqref{diff_coeff_Fourier} are chosen as follows:
\begin{equation} \label{eq:exp2:alphas}
  \alpha_1 = 0.498\ \ 
  \text{and}\ \ 
  \alpha_m = \bar\alpha m^{-1}\ \text{for $m=2, \dots, M$  with $ \bar\alpha = 0.547$}.
\end{equation}
Indeed, differentiating the diffusion coefficient with respect to parameters, we obtain
\begin{equation*}
    \frac{\partial^\kk a(\cdot, \y)}{\partial \y^\kk} =
    a(\cdot, \y) \prod_{m=1}^Mh^{k_m}_m(x).
\end{equation*}
Thus, 
\begin{equation*} 
   \left\lVert
   a^{-1}(\cdot, \y) \, \frac{\partial^\kk a(\cdot, \y)}{\partial \y^\kk}
   \right\rVert_{L^{\infty} (D)} =
   \left\lVert\ \prod_{m=1}^Mh^{k_m}_m(x) \right\rVert_{L^{\infty} (D)} \le \prod_{m=1}^M\alpha^{k_m}_m \le
   (2\dd)^{-\kk} \, \kk!
\end{equation*}
for $\alpha_m$ given in~\eqref{eq:exp2:alphas} and for some vector $\dd > \1$,
as required by~\eqref{eq:nonaffine}.

\begin{figure}[!thp]
\begin{tikzpicture}
\pgfplotstableread{sl_sc_test2Leja.dat}{\one}
\begin{loglogaxis}
[
width = 7.5cm, height = 7.5cm,						
xlabel={degrees of freedom (dof)}, 					
xlabel style={font=\fontsize{10pt}{12pt}\selectfont},		
ylabel={\rev{weighted sum of error indicators}},				
ylabel style={font=\fontsize{10pt}{12pt}\selectfont}, 		
ymajorgrids=true, xmajorgrids=true, grid style=dashed,	
xmin=65, xmax=10107734,						
ymin = 2*10^(-3),	 ymax = 0.2,							
legend style={legend pos=south west, legend cell align=left, fill=none}
]
\addplot[black,mark=o,mark size=2.2pt]		table[x=dofs, y=error]{\one};
\addplot[blue,mark=square,mark size=2pt]		table[x=dofs, y=error_s]{\one};
\addplot[red,mark=diamond,mark size=2.6pt]		table[x=dofs, y=error_p]{\one};
\legend{
$\bar\eta_\ell$ (total),
$\bar\mu_\ell$ (spatial),
$\bar\tau_\ell$ (parametric)
}
\end{loglogaxis}
\end{tikzpicture}
\quad
\begin{tikzpicture}
\pgfplotstableread{sl_sc_test2CC.dat}{\two}
\begin{loglogaxis}
[
width = 7.5cm, height = 7.5cm,						
xlabel={degrees of freedom (dof)}, 					
xlabel style={font=\fontsize{10pt}{12pt}\selectfont},		
ylabel={\rev{weighted sum of error indicators}},				
ylabel style={font=\fontsize{10pt}{12pt}\selectfont}, 		
ymajorgrids=true, xmajorgrids=true, grid style=dashed,	
xmin=65, xmax=10107734,						
ymin = 2*10^(-3),	 ymax = 0.2,							
legend style={legend pos=south west, legend cell align=left, fill=none}
]
\addplot[black,mark=o,mark size=2.2pt]		table[x=dofs, y=error]{\two};
\addplot[blue,mark=square,mark size=2pt]		table[x=dofs, y=error_s]{\two};
\addplot[red,mark=diamond,mark size=2.6pt]		table[x=dofs, y=error_p]{\two};
\legend{
$\bar\eta_\ell$ (total),
$\bar\mu_\ell$ (spatial),
$\bar\tau_\ell$ (parametric)
}
\end{loglogaxis}
\end{tikzpicture}
\caption{Test case II: evolution of the \rev{weighted sums of error indicators} for Leja (left) and CC (right) points.
The axes limits  are identical in the left and right plots.}
\label{fig:test2:indicators}
\end{figure}

\begin{figure}[!thp]
\begin{tikzpicture}
\pgfplotstableread{sl_sc_test2Leja.dat}{\one}
\begin{loglogaxis}
[
width = 7.5cm, height = 7.5cm,						
xlabel={degrees of freedom (dof)}, 					
xlabel style={font=\fontsize{10pt}{12pt}\selectfont},		
ylabel={estimated error},				
ylabel style={font=\fontsize{10pt}{12pt}\selectfont}, 		
ymajorgrids=true, xmajorgrids=true, grid style=dashed,	
xmin=65, xmax=10107734,						
ymin = 4*10^(-4),	 ymax = 0.1,							
legend style={legend pos=south west, legend cell align=left, fill=none}
]
\addplot[black,mark=o,mark size=2.2pt]		table[x=dofs, y=error_d]{\one};
\addplot[blue,mark=square,mark size=2pt]		table[x=dofs, y=error_s_d]{\one};
\addplot[red,mark=diamond,mark size=2.6pt]		table[x=dofs, y=error_p_d]{\one};
\legend{
$\eta_\ell$ (total),
$\mu_\ell$ (spatial),
$\tau_\ell$ (parametric)
}
\end{loglogaxis}
\end{tikzpicture}
\quad
\begin{tikzpicture}
\pgfplotstableread{sl_sc_test2CC.dat}{\two}
\begin{loglogaxis}
[
width = 7.5cm, height = 7.5cm,						
xlabel={degrees of freedom (dof)}, 					
xlabel style={font=\fontsize{10pt}{12pt}\selectfont},		
ylabel={estimated error},				
ylabel style={font=\fontsize{10pt}{12pt}\selectfont}, 		
ymajorgrids=true, xmajorgrids=true, grid style=dashed,	
xmin=65, xmax=10107734,						
ymin = 4*10^(-4),	 ymax = 0.1,							
legend style={legend pos=south west, legend cell align=left, fill=none}
]
\addplot[black,mark=o,mark size=2.2pt]		table[x=dofs, y=error_d]{\two};
\addplot[blue,mark=square,mark size=2pt]		table[x=dofs, y=error_s_d]{\two};
\addplot[red,mark=diamond,mark size=2.6pt]		table[x=dofs, y=error_p_d]{\two};
\legend{
$\eta_\ell$ (total),
$\mu_\ell$ (spatial),
$\tau_\ell$ (parametric)
}
\end{loglogaxis}
\end{tikzpicture}
\caption{Test case II: evolution of the error estimates for Leja (left) and CC (right) points.
The axes limits  are identical in the left and right plots.}
\label{fig:test2:errors}
\end{figure}
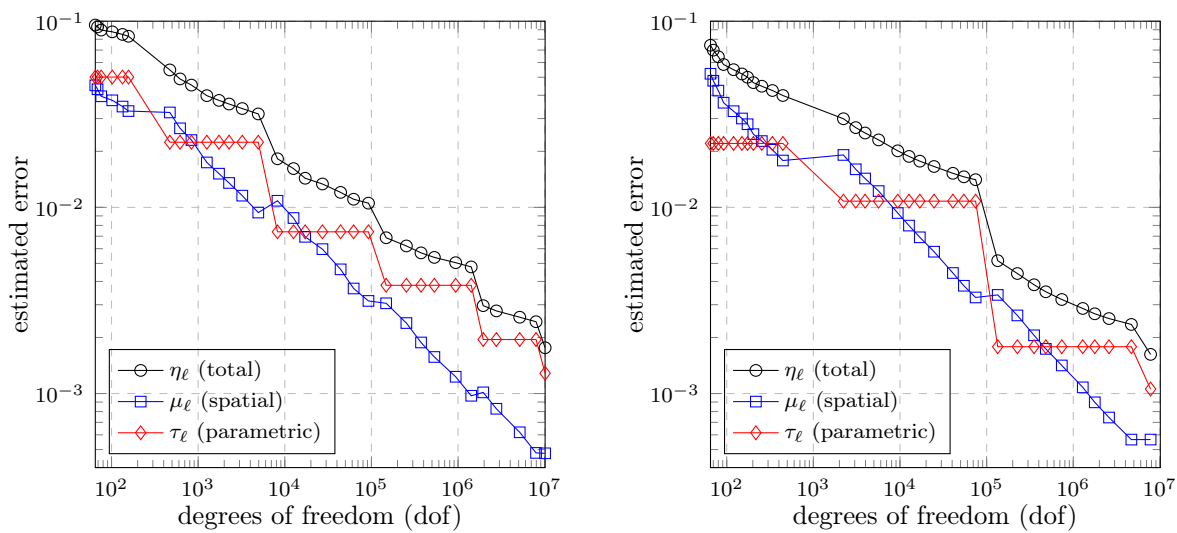

For this test case, we set the dimension of the parameter domain to $M = 4$ and run Algorithm~\ref{algorithm}
with $\tt{error tolerance} =$ {\tt 2e-3}.
This tolerance was reached after 31 iterations (including 5 parametric refinement steps) for Leja points and
after 31 iterations (including 3 parametric enrichments) for CC points.
We record the evolution of
\rev{the weighted sums of} error indicators
as well as the evolution of the corresponding error estimates;
see Figures~\ref{fig:test2:indicators} and~\ref{fig:test2:errors}.
From these plots we can see that both the \rev{weighted sums of} error indicators and the error estimates show similar behavior
to that observed in Test Case~I.
Additionally, for both test cases, we observe that the execution of the algorithm with Leja collocation points
requires more parametric enrichments than that with CC points to reach the same tolerance
\rev{(6 vs. 4 enrichments for test case~I and 5 vs. 3 enrichments for test case~II).
As a consequence, the overall computational time is increased when using Leja points compared to using CC points
(8216 seconds vs. 7757 seconds for test case~I and 5427 seconds vs. 3936 seconds for test case~II).}
This is explained by the fact that every new multi-index generates more collocation points of the CC type than those of the Leja type.
For instance, in 1D, each index $i \in \N$ generates one new Leja collocation point and $2^{i-1}$ new CC~points.

\rev{
\section{Concluding remarks} \label{sec:conclusions}

Adaptive algorithms provide effective solution strategies for high-dimensional parametric PDE problems.
They generate accurate and fast-converging approximations that resolve local spatial features and
adapt to the parametric anisotropy of the PDE solution.
While many adaptive algorithms have been designed and implemented in this context,
the mathematical analysis of these solution strategies is much less developed.
In this paper, for a model parametric PDE problem, we have performed the convergence analysis of an adaptive SC-FEM algorithm
driven by hierarchical a posteriori error indicators proposed in~\cite{Bespalov22}.
Our main result in Theorem~\ref{theorem:main} provides a theoretical guarantee that, for any given positive tolerance,
the proposed adaptive algorithm terminates after a finite number of iterations.
Our theoretical results are valid for spatial domains in $\R^2$ or $\R^3$ and
for affine or non-affine finite-dimensional parametrization of PDE inputs.

In this work, we have employed the \emph{single-level} (rather than \emph{multilevel}) construction of SC-FEM approximations
that assigns the same finite element space to all collocations points.
This choice is primarily motivated by the results of numerical experiments in~\cite{Bespalov22, BS23}.
These results have indicated that the single-level version is likely
to be more efficient when the same adaptively refined finite element mesh can adequately resolve solution features
for a range of individually sampled problems, which is often the case for the model parametric problem~\eqref{eq:pde:strong}.
While the extension of our convergence analysis to the adaptive multilevel SC-FEM algorithm proposed in~\cite{BS23}
is of interest from the theoretical point of view,
this extension is nontrivial due to a different marking strategy (see~\cite[Algorithm~2]{BS23})
and the need to incorporate an additional adaptive strategy for defining suitable meshes for newly added collocation points (see~\cite[Algorithm~3]{BS23}).

Other possible extensions of this work include: (i)~the parametrization of PDE inputs in terms of
countably infinite number of random parameters (i.e., $M = \infty$) and the associated algorithmic aspects of dimension adaptivity
(see, e.g.,~\cite[section~7]{GuignardN18}); and
(ii)~the important yet challenging case of unbounded random parameters
that appear, e.g., in parametric representations of log-normal random fields.
The progress in these directions hinges on finding reliable a posteriori error estimates and appropriate error indicators and
will require a nontrivial extension of the analysis and algorithmic developments in~\cite{Bespalov22}.

}

\bibliographystyle{alpha}
\bibliography{references}
\end{document}